%% file: revised_paper.tex
\documentclass[final]{siamart1116}

\usepackage{amsmath}
\usepackage{mathtools}
\usepackage{amsfonts}
\usepackage{amssymb}
\usepackage{fancyhdr}
\usepackage{graphicx}
\usepackage[section]{placeins}
\usepackage{nicefrac}
\usepackage{bm}
\usepackage{xcolor}
\usepackage[format=plain,indention=.5cm, font={it,small},labelfont=bf]{caption}
\usepackage{subcaption}
\usepackage{float}
\usepackage{enumerate}
\usepackage{tikz}
\usetikzlibrary{arrows.meta}
\usepackage[all]{xy}
\newsiamthm{clm}{Claim}
\newsiamremark{rmk}{Remark}
\newsiamremark{note}{NOTE}
\numberwithin{theorem}{section}

\newcommand{\ol}{\overline}
\renewcommand{\b}{\bm}

\definecolor{dgreen}{RGB}{49,128,23}
\definecolor{nicepink}{RGB}{255, 0, 102}
\definecolor{nicered}{RGB}{255, 80, 80}

\newcommand{\bR}{\mathbb{R}}

\newcommand{\bZ}{\mathbb{Z}}

\newcommand{\cD}{\mathcal{D}}
\newcommand{\cE}{\mathcal{E}}
\newcommand{\cF}{\mathcal{F}}

\newcommand{\cK}{\mathcal{K}}

\newcommand{\cN}{\mathcal{N}}

\newcommand{\cP}{\mathcal{P}}

\newcommand{\cR}{\mathcal{R}}
\newcommand{\cS}{\mathcal{S}}
\newcommand{\cT}{\mathcal{T}}

\newcommand{\cV}{\mathcal{V}}

\newcommand{\ka}{\kappa}
\newcommand{\fp}{\varrho}

\author{James D. Brunner \and Gheorghe Craciun}
\title{Robust persistence and permanence of polynomial and power law dynamical systems}

\headers{Robust persistence and permanence of polynomial systems}{J. D. Brunner and G. Craciun}
\begin{document}
\maketitle
\slugger{siap}{xxxx}{xx}{x}{x--x}

\begin{abstract}
	A persistent dynamical system in $\bR^d_{> 0}$ is one whose solutions have positive lower bounds for large $t$, while a permanent dynamical system in $\bR^d_{> 0}$ is one whose solutions have uniform upper and lower bounds for large $t$. These  properties have important applications for the study of mathematical models in biochemistry, cell biology, and ecology. Inspired by reaction network theory, we define a class of polynomial dynamical systems called tropically endotactic. We show that two-dimensional tropically endotactic polynomial dynamical systems are permanent, irrespective of the values of (possibly time-dependent) parameters in these systems. These results generalize the permanence of two-dimensional reversible, weakly reversible, and endotactic mass action systems.
\end{abstract}

\section{Introduction}\label{intro}
\input{introduction_rev}

\section{Differential inclusions}\label{defnot}

\input{notation_rev}

\section{Persistence of tropically endotactic differential inclusions}\label{proof_1}
\input{main_proof_rev}
\section{Polynomial dynamical systems with variable coefficients}\label{applemmas}
\input{defs_vk_rev}
\section{Example systems}\label{examples}
\input{examples_rev}

\section{Future Work}

\input{furture_work_rev}

\section{Acknowledgments}
\input{acks_rev}

\bibliographystyle{siamplain}
\bibliography{../../../persistencebib}
\end{document}

%% file: introduction_rev.tex
Polynomial dynamical systems are used to model many physical, chemical, and biological processes. For example, systems of polynomial differential equations have come into use in the modeling and simulation of large biochemical networks, population dynamics, and epidemiology \cite{ekesh}\cite{murray}. The study of chemical and biochemical reaction network models is especially concerned with dynamical systems which have polynomial right-hand sides \cite{mfeinlec}. 

In order to understand global long term behavior of solutions of polynomial dynamical systems, we seek to determine whether or not solutions with positive initial conditions remain bounded, and bounded away from zero. A dynamical system on $\bR^d_{>0}$ is said to be \emph{persistent} if for any solution $\b{x}(t)$ with $\b{x}(0)\in \bR^d_{>0}$, we have $\liminf_{t\rightarrow\infty}x_i(t) > 0$ for all $i\in \{1,...,d\}$. In the context of population modeling, this means that no species becomes extinct. The stronger property \emph{permanence} means that there exists a compact region $\cR$ which does not intersect $\partial\bR^d_{>0}$ such that any solution $\b{x}(t)$ with $\b{x}(0)\in \bR^d_{>0}$ ultimately resides inside $\cR$. In other words, there exists a compact attracting region in $\bR^d_{>0}$. Clearly, permanence implies persistence, and additionally it implies that solutions are uniformly bounded, and uniformly bounded away from $\partial \bR^d_{>0}$. 


In recent work, it has been shown that two dimensional weakly reversible and endotactic polynomial dynamical systems are permanent \cite{cranaz}. We will extend these results to the larger class of tropically endotactic polynomial and power law dynamical systems. This class of dynamical systems has the advantage of being robust with respect to changes in the parameters of the systems, which is often useful in applications, because, in practice, it is often difficult or impossible to measure these parameters accurately. In future work, we will show that the tropically endotactic condition is very close to being necessary and sufficient for permanence. 

Other results about persistence of polynomial dynamical systems that result from chemical reaction networks have been obtained by using Petri net-based methods \cite{sontagangelideleenher}. Furthermore, results about the global convergence properties of solutions to polynomial dynamical systems that result from reaction networks have been obtained by taking advantage of special properties of these networks\cite{dave1}\cite{daveandann}\cite{murad2}\cite{toricsystems}\cite{hornandfienberg1}\cite{anneGAC}\cite{Johnston2016}\cite{casianPers}. 

\bigskip
In general, polynomial dynamical systems
\footnote{Note that if we restrict $\b{s}_i\in \bZ^d_{\geq 0}$, then \cref{pds} is exactly the set of polynomial dynamical systems. In this paper we allow the more general case $\b{s}_i\in \bR^d$, often called \emph{power law} systems.} have the form
\begin{equation}\label{pds}
\dot{\b{x}}  = \sum_{i = 1}^n \b{x}^{\b{s}_i}\b{v}_i
\end{equation}
where $\b{x} = (x_1,...,x_d) \in \bR^d$, $\b{s}_i \in \bR^d$, and $\b{v}_i \in \bR^d$. Vector exponentials in $\bR^d$ of the form $\b{x}^{\b{s}}$ are defined by
\begin{equation}
\b{x}^{\b{s}}  = \prod_{j=1}^d x_j^{s_j}
\end{equation}
Here, we are concerned with the more general class of dynamical systems of the form
\begin{equation}
\dot{\b{x}}  = \sum_{i = 1}^n \ka_i(t) \b{x}^{\b{s}_i}\b{v}_i,
\end{equation}
where we allow the coefficients $\ka_i(t) \in \bR$ to vary in time, but we assume that there exist some $\varepsilon>0$ such that $\varepsilon < \ka_i(t) < \frac{1}{\varepsilon}$ for all $t>0$.  We refer to such systems as \emph{variable $\ka$ polynomial (v$\ka$-polynomial) dynamical systems}.

In the analysis of v$\ka$-polynomial dynamical systems, we make use of a special class of \emph{differential inclusions}. In general, differential inclusions are dynamical systems of the form
\begin{equation}\label{diffincintro}
\b{\dot{x}} \in \b{F}(\b{x})
\end{equation}
where $\b{F}(\b{x})$ is a set-valued map. Here, we introduce a class of differential inclusions which captures the dynamics of the polynomial dynamical systems we are interested in. These differential inclusions are called \emph{$\cN$-cone differential inclusions}, where $\widehat{\cN} = \log(\cN)$ is a complete fan in $\bR^d$ (see \cref{fandef,expfan,ncone,fattening}). $\cN$-cone differential inclusions are piecewise constant in $\bR^d_{>0}$ on regions determined by $\widehat{N}\in \widehat{\cN}$. Furthermore they are autonomous and consist of a convex cone $K(N)$ at each point $\b{x}$ in the region determined by $\widehat{N}$. If $\b{f}(\b{x},t) \in \b{F}(\b{x})$ for all $\b{x}$ and $t$ we say that the dynamical system $\b{\dot{x}} = \b{f}(\b{x},t)$ is embedded in the differential inclusion $\b{\dot{x}} \in \b{F}(\b{x})$. Clearly, if a dynamical system is embedded in a differential inclusion, then solutions of the dynamical system are also solutions of the differential inclusion. In particular, if the solutions of a differential inclusion are persistent or permanent, then the same is true for the solutions of a dynamical system embedded in it.

We show that if a two dimensional $\cN$-cone differential inclusion $\b{\dot{x}} \in \b{F}(\b{x})$ is tropically endotactic (see \cref{tropendo}) then its solutions remain bounded and do not approach $\partial \bR^2_{>0}$. This implies that if a v$\ka$-polynomial dynamical system can be embedded into a tropically endotactic differential inclusion, then its solutions are also bounded and do not approach $\partial \bR^2_{>0}$. Moreover, we then show that if this embedding is strict (see \cref{strictemb}) then the polynomial dynamical system has the stronger property of \emph{permanence}. Finally, we give examples of polynomial dynamical systems which are embedded in tropically endotactic differential inclusions.

\bigskip
More specifically, in order to study the persistence of an $\cN$-cone differential inclusion, we identify sets of ``escape directions" $B^{\delta}(N)$ for each set $N\in \cN$. Informally speaking, these are directions along which trajectories may escape any compact region that does not intersect $\partial \bR^d_{>0}$ while staying inside $N$ (see \cref{edirs,escapedirsfig} for details). For example, if the closure of $N$ contains the $y$-axis, $B^{\delta}(N)$ contains any vector $\b{v}$ such that $\b{v} \cdot (1,0) < 0$ (i.e., the left half plane).

We define \emph{tropically endotactic differential inclusions} by comparing the cones $K(N)$ of an $\cN$-cone differential inclusion to the escape directions $B^{\delta}(N)$ corresponding to $\cN$. We call a differential inclusion \emph{tropically endotactic} when $K(N)$ does not intersect the interior of $B^{\delta}(N)$\footnote{along with a technical condition on the cones $K(N)$ and $K(M)$ when $\widehat{M}$ is a face of $\widehat{N}$} for all $N \in \cN$ (see \cref{tropendo} and for an example see \cref{permsyspartition}). This condition is easy to check in two dimensions. Furthermore, it provides a sufficient condition for persistence of solutions of a differential inclusion:
\begin{theorem}\label{introthm1}
	Let $\b{\dot{x}}\in F(\b{x})$ be a differential inclusion defined on $\bR^2_{> 0}$. If $\b{\dot{x}}\in F(\b{x})$ is tropically endotactic, then it is persistent and has bounded trajectories.
\end{theorem}

Therefore, we can use a tropically endotactic differential inclusion to conclude persistence of a v$\ka$-polynomial dynamical system which is embedded in it. We can also obtain a stronger result when the embedding is into the interiors of the sets of the differential inclusion. If $\b{\dot{x}} = \b{f}(\b{x},t)$ has the property that $\b{f}(\b{x},t) \in \b{F}(\b{x})^{\circ}$ for every $\b{x}$ and $t$, where $\b{F}(\b{x})^{\circ}$ is the interior of $\b{F}(\b{x})$, then we say that $\b{\dot{x}} = \b{f}(\b{x},t)$ is \emph{strictly embedded} in the differential inclusion $\b{\dot{x}} \in \b{F}(\b{x})$. If a v$\ka$-polynomial dynamical system is strictly embedded in a tropically endotactic differential inclusion, we call it a \emph{tropically endotactic v$\ka$-polynomial dynamical system}. We prove the following theorem, which states that being tropically endotactic is a sufficient condition for permanence of a (possibly non-autonomous) v$\ka$-polynomial dynamical system:
\begin{theorem}\label{introthm2}
	Any two-dimensional tropically endotactic v$\ka$-polynomial dynamical system is permanent.
\end{theorem}
 


Finally, we give examples of systems which are not endotactic but which are tropically endotactic, and therefore permanent. We also show that \cref{introthm2} is a generalization of the permanence of weakly reversible two dimensional systems as described in \cite{gheorgheGAC}.

%% file: notation_rev.tex
\subsection{Piecewise Constant Cone Differential Inclusions}

A differential inclusion is a dynamical system 
\[
\b{\dot{x}} \in \b{F}(\b{x})
\]
where $\b{F}$ is a set-valued map. We are interested in the special case in which $\b{F}(\b{x})$ is a cone, and is constant on regions of $\bR^2_{> 0}$. 

We use a \emph{fan} in $\bR^2$ to define a cover of $\bR^2$ and $\bR^2_{>0}$. Definitions of fans and cones follow \cite{polytopes} and \cite{nonneg}. With a set $S \subseteq \bR^d$ we associate the set $\mathit{Cone}(S)$, the cone \emph{generated} by $S$, which we define as the closure of the set of all finite, nonnegative linear combinations of the elements of $S$ \cite{nonneg}
\footnote{We use the definition from \cite{nonneg} and \cite{polytopes}, which defines cones to be closed and convex. In other sources, cones are not necessarily convex.}.

We will be concerned with the cones generated by finite sets of vectors $S = \{\b{v}_1,...,\b{v}_k\}$. In this case, a cone $K$ is
\[
K = \left\{\left.\b{w} = \sum_{i=1}^k a_i \b{v}_i\right| a_i\geq 0\right\}
\]
and is called a \emph{polyhedral cone}. In what follows, we will simply use the word ``cone" to mean polyhedral cone.

A cone $K$ is \emph{solid} if the interior of $K$ is non-empty. A \emph{supporting hyperplane} $H$ of a cone $K$ is a hyperplane that intersects $K$ at the origin and such that $K$ is contained in only one of the two half-spaces determined by $H$. A \emph{face} of $K$ is the intersection between $K$ and a supporting hyperplane.

\begin{definition}\label{fandef}\cite{polytopes}
A \emph{fan} (or polyhedral fan) in $\bR^d$ is a finite family $\widehat{\cN} = \{\widehat{N}_1,\widehat{N}_2,...,\widehat{N}_n\}$ of nonempty polyhedral cones such that:
\begin{enumerate}[(i)]
\item Every nonempty face of a cone in $\widehat{\cN}$ is also a cone in $\widehat{\cN}$.
\item The intersection of any two cones in $\widehat{\cN}$ is a face of both.
\end{enumerate}
\end{definition}

If the union of cones in the fan $\cN$ is $\bR^d$, then $\cN$ is called a \emph{complete fan}. Because we are concerned only with complete fans, we will use the word fan to mean complete fan. A complete fan is a \emph{cover} of $\bR^d$, and moreover, the collection of relative interiors of the cones of the fan forms a partition of $\bR^d$. We can use a cover of $\bR^d$ given by a complete fan to construct a cover of $\bR^d_{>0}$, as described below.
\begin{definition}\label{expfan}
A finite family of sets $\cN = \{N_1,N_2,...,N_k\}$ in $\bR^d_{>0}$ is called an \emph{exponential fan} if there exists a complete fan $\widehat{\cN} = \{\widehat{N}_1,\widehat{N}_2,...,\widehat{N}_k\}$ in $\bR^d$ such that $N_i = \exp(\widehat{N}_i)$ for $i=1,...,k$.
\end{definition}
Notice that, because the map $\exp:\bR^d\rightarrow \bR^d_{>0}$ is bijective, an exponential fan defines a cover of $\bR^d_{>0}$, and the collection of relative interiors form a partition of $\bR^2_{>0}$. See \cref{fattening} (a) and (b) for examples of a fan and exponential fan, respectively. 

\medskip

Given an exponential fan $\cN$ of $\bR^2_{>0}$, before we can define an $\cN$-cone differential inclusion, we must define different covers $\mathit{fat}_{\fp}(\widehat{\cN})$ and $\mathit{fat}_{\fp}(\cN)$ of $\bR^2$ and $\bR^2_{> 0}$ respectively, for some number $\fp \in (0,1)$, as in \cref{fattening} (c) and (d). First, we define a compact region around the origin
\begin{multline}
\mathit{fat}_{\fp}(\b{0})  = \left\{\b{X} \in \bR^2|dist(\b{X},\widehat{N}_{j})\leq |\log(\fp)|,\,\text{for at least}\right. \\ \left. \text{ two one-dimensional cones }\widehat{N}_{j}\right\}
\end{multline}
Then, we ``fatten" the one dimensional cones of the fan $\widehat{\cN}$. Denote by $R^{\circ}$ the interior of a region $R \subset \bR^2$. For each one dimensional cone $\widehat{N}_j$ of $\widehat{\cN}$, we define
\begin{equation}
\mathit{fat}_{\fp}(\widehat{N}_{j}) = \left\{\b{X}\in \bR^2|dist(\b{X},\widehat{N}_{j})\leq |\log(\fp)|\right\}\setminus \mathit{fat}_{\fp}(\b{0})^{\circ}
\end{equation}
which is a strip centered around $\widehat{N}_{j}$ with an area near the origin removed. Finally, we also must take into account regions in the two dimensional cones not included in these strips. We thus define, for two dimensional cones $\widehat{N}_i$,
\begin{equation}
\mathit{fat}_{\fp}(\widehat{N}_{i}) = \widehat{N}_{i}\setminus \left(\bigcup_{\widehat{N}_{j}|dim(\widehat{N}_{j}) \leq 1}\mathit{fat}_{\fp}(\widehat{N}_{j})\right)^{\circ}
\end{equation}
That is, $\mathit{fat}_{\fp}(\widehat{N}_{i})$ is obtained from the original two dimensional cone $\widehat{N}_{i}$ by removing the regions $\mathit{fat}_{\fp}(\widehat{N}_j)$ for lower dimensional cones $\widehat{N}_j$ (except for the borders, so that all regions remain closed). Then, we define $\mathit{fat}_{\fp}(\widehat{\cN}) = \{\mathit{fat}_{\fp}(\widehat{N}_i)| \widehat{N}_i \in \widehat{\cN}\}$. See \cref{fattening} (c) for examples of the regions $\mathit{fat}_{\fp}(\widehat{N}_i)$.

\begin{figure}[H]
	\begin{center}
		\begin{subfigure}[t]{0.22\textwidth}
			\begin{center}
				\begin{tikzpicture}[scale = 0.5]
				\draw[<->] (0,-2.3) -- (0,2.3);
				\draw[<->] (-2.3,0) -- (2.3,0);
				\draw[color = blue, very thick] plot[domain = -1:0] (\x,2*\x);
				\draw[color = blue, very thick] plot[domain = -1:0] (2*\x,\x);
				\draw[color = blue, very thick] plot[domain = 2:0] (-1*\x,\x);
				\draw[color = blue, very thick] plot[domain = 1:0] (\x,2*\x);
				\draw[color = blue, very thick] plot[domain = 0.67:0] (3*\x,\x);
				\draw[color = blue, very thick] plot[domain = 1:0] (\x,-2*\x);
				\end{tikzpicture}
			\end{center}
			\caption{}
		\end{subfigure}\quad
		\begin{subfigure}[t]{0.22\textwidth}
			\begin{center}
				\begin{tikzpicture}[scale = 0.8]
				\draw[->] (0,0) -- (0,3);
				\draw[->] (0,0) -- (3,0);
				\draw[color = blue, very thick] plot[domain = 1:0] (\x,\x^2);
				\draw[color = blue, very thick] plot[domain = 1:0] (\x^2,\x);
				\draw[color = blue, very thick] plot[domain = 1:0.35] (\x,\x^-1);
				\draw[color = blue, very thick] plot[domain = 1.7:1] (\x,\x^2);
				\draw[color = blue, very thick] plot[domain = 1.4:1] (\x^3,\x);
				\draw[color = blue, very thick] plot[domain = 3:1] (\x,\x^-2);
				\end{tikzpicture}
			\end{center}
			\caption{}
		\end{subfigure}\quad
		\begin{subfigure}[t]{0.22\textwidth}
			\begin{center}
				\begin{tikzpicture}[scale = 0.5]
				\draw[<->] (0,-2.3) -- (0,2.3);
				\draw[<->] (-2.3,0) -- (2.3,0);
				\draw[color = blue, line width = 0.2 cm] plot[domain = -1:0] (\x,2*\x);
				\draw[color = blue, line width = 0.2 cm] plot[domain = -1:0] (2*\x,\x);
				\draw[color = blue, line width = 0.2 cm] plot[domain = 2:0] (-1*\x,\x);
				\draw[color = blue, line width = 0.2 cm] plot[domain = 1:0] (\x,2*\x);
				\draw[color = blue, line width = 0.2 cm] plot[domain = 0.67:0] (3*\x,\x);
				\draw[color = blue, line width = 0.2 cm] plot[domain = 1:0] (\x,-2*\x);
				\fill[color = red] (0,-0.4)--(-0.37,-0.37)--(-0.27,0.05)--(-0.055,0.28)--(0.33,0.3)--(0.23,-0.08);
\end{tikzpicture}
			\end{center}
			\caption{}
		\end{subfigure}\qquad
		\begin{subfigure}[t]{0.22\textwidth}
			\begin{center}
				\begin{tikzpicture}[scale = 0.8]
				\draw[->] (0,0) -- (0,3);
				\draw[->] (0,0) -- (3,0);
				\path[fill,color = blue] plot[domain = 1:0] (\x,1.1*\x^2) -- plot[domain = 0:1] (\x,0.9*\x^2);
				\path[fill,color = blue] plot[domain = 1:0] (1.1*\x^2,\x) -- plot[domain = 0:1] (0.9*\x^2,\x);
				\path[fill,color = blue] plot[domain = 1:0.35] (\x,1.1*\x^-1) -- plot[domain = 0.35:1] (\x,0.9*\x^-1);
				\path[fill,color = blue] plot[domain = 1.7:1] (\x,1.1*\x^2) -- plot[domain = 1:1.7] (\x,0.9*\x^2);
				\path[fill,color = blue] plot[domain = 1.4:1] (1.1*\x^3,\x) -- plot[domain = 1:1.4] (0.9*\x^3,\x);
				\path[fill,color = blue] plot[domain = 3:1] (\x,1.1*\x^-2) -- plot[domain = 1:3] (\x,0.9*\x^-2);
				\fill[color = red,rotate=45] (1.4,0) ellipse(4 pt and 2 pt);
				\end{tikzpicture}
			\end{center}
			\caption{}
		\end{subfigure}
	\end{center}
	\caption{\emph{(a)} A fan $\widehat{\cN}$, which contains 13 cones $\widehat{N}_i$; 6 of these are one dimensional, 6 are two dimensional, and one (the origin) is zero dimensional. \emph{(b)}The exponential fan $\cN = \exp(\widehat{\cN})$. \emph{(c)}The regions $\mathit{fat}_{\fp}(\widehat{N}_i)$. The regions corresponding the two dimensional cones are shown in white, those corresponding to one dimensional cones are shown in blue, and $\mathit{fat}_{\fp}(\b{0})$ is shown in red. \emph{(d)}The regions $\mathit{fat}_{\fp}(N_i) = \exp(\mathit{fat}_{\fp}(\widehat{N}_i))$.}
	\label{fattening}
\end{figure}
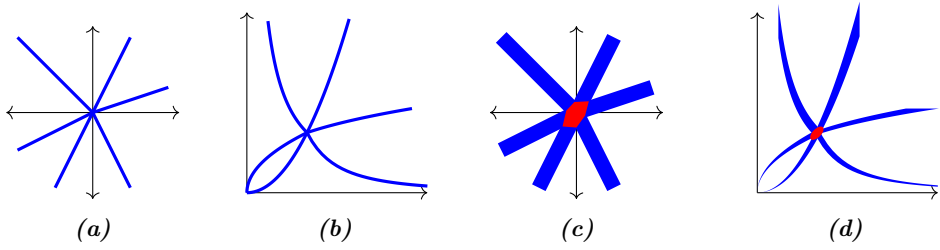

These regions can be used to construct a cover of $\bR^2_{> 0}$. That is, if $\widehat{\cN} = \{\widehat{N}_i\}$ is a complete fan and $\cN = \exp(\widehat{\cN})$ is its corresponding exponential fan, then we define
\begin{equation}
\mathit{fat}_{\fp}(N_{i}) = \left\{\b{x}| \log(\b{x}) \in \mathit{fat}_{\fp}(\widehat{N}_{i})\right\}
\end{equation}
where $N_i = \exp(\widehat{N}_i)$, and we define $\mathit{fat}_{\fp}(\cN) = \{\mathit{fat}_{\fp}(N_i)|N_i \in \cN\}$. See \cref{fattening} (d) for examples of the regions $\mathit{fat}_{\fp}(N_i)$.

Finally, we can now define the class of differential inclusions of interest.
\begin{definition}\label{ncone}
	Consider an exponential fan $\cN$ in $\bR^d_{>0}$. An \emph{$\cN$-cone differential inclusion $\cK(\cN)$ (with parameter $\fp$)} is a dynamical system of the form
	\[
	\b{\dot{x}} \in \bigcup_{\{N|\b{x} \in \mathit{fat}_{\fp}(N)\}} K(N)
	\]
	for some $\fp >0$, where $K(N)$ is a cone for each $N\in \cN$.
\end{definition}

Differential inclusions provide a framework for sacrificing precise information about a dynamical system in order to simplify the analysis in some way. We are interested in analyzing polynomial dynamical systems which may be non-autonomous and highly nonlinear. Such systems are in general difficult to analyze. We therefore replace these systems with $\cN$-cone differential inclusions, which are autonomous and piecewise constant. Then, if we have that the right-hand side of the non-autonomous polynomial dynamical system is contained in the right-hand side of the $\cN$-cone differential inclusion, the properties of solutions to the differential inclusion will be satisfied by solutions to the polynomial dynamical system. In order to make this rigorous, we define a notion of \emph{embedding}.

\begin{definition}
	We say a dynamical system $\b{\dot{x}} = \b{f}(\b{x},t)$ is \emph{embedded} into the differential inclusion $\b{\dot{x}}\in \b{F}(\b{x})$ in the domain $\Omega$ if $\b{f}(\b{x},t) \in \b{F}(\b{x})$ for every $\b{x} \in \Omega$ and for all $t$.	
\end{definition}

Similarly, we can embed one differential inclusion into another.

\begin{definition}
	We say a differential inclusion $\b{\dot{x}} \in \b{G}(\b{x})$ is \emph{embedded} into the differential inclusion $\b{\dot{x}}\in \b{F}(\b{x})$ in the domain $\Omega$ if $\b{G}(\b{x}) \subseteq \b{F}(\b{x})$ for every $\b{x} \in \Omega$.	
\end{definition}

We will need a stronger notion of embedding to obtain some conclusions about the systems of interest.
\begin{definition}\label{strictemb}
	We say a dynamical system $\b{\dot{x}} = \b{f}(\b{x},t)$ is \emph{strictly embedded} into the differential inclusion $\b{\dot{x}}\in \b{F}(\b{x})$ in the domain $\Omega$ if $\b{f}(\b{x},t) \in \b{F}(\b{x})^{\circ}$ for every $\b{x} \in \Omega$ and for all $t$.
\end{definition}

\subsection{Persistence}
We are concerned with the long term behavior of solutions to these differential inclusions. In particular, we want to determine if an $\cN$-cone differential inclusion allows solutions which reach $\partial \bR^2_{> 0}$. If the differential inclusion does not allow such solutions, we call it persistent.

\begin{definition}\label{defper}
	A d-dimensional dynamical system is called \emph{persistent} on $\bR^d_{\geq 0}$ if for any solution $\b{x}(t)$ defined on an interval $I$ containing $t = 0$  with initial condition $\b{x}_0 \in \bR^d_{>0}$, there exists some $\varepsilon>0$ such that we have
		\[
		x_i(t) > \varepsilon 
		\]
	for all $i \in \{1,...,d\}$ and for all  $t\in I \cap [0,\infty)$.\cite{cranaz}
\end{definition}

%

In particular, we will say that an $\cN$-cone differential inclusion is persistent if all absolutely continuous functions satisfying the differential inclusion \cite{smirnov} have the above property.

\subsection{Escape directions}
Let $\widehat{\cN}$ be a complete fan and $\cN = \exp(\widehat{\cN})$ be its associated exponential fan. We will construct a cone $B^{\delta}(N)$ for each $N\in \cN$, referred to as the ``escape directions of $N$". Consider curves $C(t)$ which have the property that for any compact set $\cR \subset \bR^2_{>0}$, there exists some $t_{\cR}$ such that for $t>t_{\cR}$, $C(t) \not \in \cR$. We will use the notation $C(t) \rightarrow \partial\bR^d_{> 0}$ for curves with this property. Notice that curves of the form $\exp(\b{r}t)$ with $\b{r} \neq 0$ satisfy $\exp(\b{r}t) \to \partial \bR^2_{> 0}$. The curves $C(t)$ described below provide a more general way to leave $\cR$ while staying in or close to $N$. Also, we use the notation $[\b{w}]_n$ to denote the unit vector in the direction of a vector $\b{w} \in \bR^2$ and $C'(t)$ to mean the tangent direction to $C$. 

%
More precisely, we consider $C(t) = \exp\left(\b{r}t + g(t) \b{p}\right)$ where $\b{r} \in \widehat{N}$, $\|\b{r}\|=1$, $g(t) = 1 - \alpha e^{-\beta t}$, $\beta\geq 0$, $\alpha \in \bR$, and $\b{p}\in \b{r}^{\perp}$. If $C(t)$ is such a curve and there exists some $t_0$ such that $\log(C(t))\subset \mathit{fat}_{\fp}(\widehat{N})$ for $t>t_0$, we call $C(t)$ a \emph{$\fp$-escape curve of $N$}. Then, we define the set of $\delta$-escape directions for $N$ as follows.

\begin{definition}\label{edirs}
	The set of \emph{$\delta$-escape directions} for a region $N\in \cN$ and $\delta >0$ is the cone
	\begin{equation}
	B^{\delta}(N) = \bigcap_{\fp < 1} \mathit{Cone}\left(\left\{\left. \left[C'(t)\right]_n\right|C(t)\text{ is a \emph{$\fp$-escape curve of }}N,\,t>\nicefrac{1}{\delta}\right\}\right)
	\end{equation}
\end{definition}
Intuitively, these are cones with the property that, for $\delta$ sufficiently small, if a curve $\b{y}(t)$ has some point $\b{y}(t_0) \in \mathit{fat}_{\fp}(N)$ and tangent $\b{y}'(t) \in B^{\delta}(N)$ for $t>t_0$, then $\b{y}(t) \rightarrow \partial \bR^2_{> 0}$ and $\b{y}(t) \in \mathit{fat}_{\fp}(N)$ for $t>t_0$. Clearly, if such a $\b{y}(t)$ is a solution to an $\cN$-cone differential inclusion, the differential inclusion is not persistent or does not have bounded trajectories. Using \cref{edirs}, these directions can be explicitly calculated. All the qualitatively different possible cones $B^{\delta}(N_i)$ are shown in \cref{escapedirsfig}. See \cref{curve_calc1,curve_calc2,curve_calc3,curve_calc4,curve_calc5} in \cref{examples} for an example of how to calculate some cones $B^{\delta}(N_i)$. Such calculations reveal that if a region $N_i$ is not adjacent to the lines $x=0$, or $y=0$, and does not contain some part of the line $y=x$, then the cone $B^{\delta}(N_i)$ approaches either a vertical or a horizontal half-line as $\delta$ approaches $0$.
\begin{figure}
		\begin{center}
			\begin{tikzpicture}[scale = 2]
			\draw[->] (0,0) -- (0,2.5);
			\draw[->] (0,0) -- (2.5,0);
			\path[fill,color = blue,opacity = 0.5] plot[domain =1.51:0] (\x,1.1*\x^2) -- plot[domain = 0:1.67] (\x,0.9*\x^2);
			\path[fill,color = blue,opacity = 0.5] plot[domain =1.51:0] (1.1*\x^2,\x) -- plot[domain = 0:1.67] (0.9*\x^2,\x);
			\path[fill,color = blue,opacity = 0.5] plot[domain =0.35:2.5] (\x,0.9*\x^-1) -- plot[domain = 2.5:0.43] (\x,1.1*\x^-1);
			\path[fill,color = blue,opacity = 0.5] (0,1.1)--(2.5,1.1)--(2.5,0.9)--(0,0.9);
			\path[fill,color = blue,opacity = 0.5] (1,1)--(2.5,2.35)--(2.5,2.5)--(2.35,2.5);
			\draw[color = red, latex-latex] (2.5,2.1)--(2.1,2.1)--(2.1,2.5);
			\fill[color = red, opacity = 0.5]  (2.5,2.1)--(2.1,2.1)--(2.1,2.5);
			\draw[color = red, -latex] (2.1,1.8)--(2.5,1.8);
			\fill[color = red, opacity = 0.5]  (2.5,1.85)--(2.1,1.8)--(2.5,1.8);
			\draw[color = red, -latex] (2.1,1.52)--(2.5,1.52);
			\fill[color = red, opacity = 0.5]  (2.5,1.57)--(2.1,1.52)--(2.5,1.52);
			\draw[color = red, -latex] (2.1,1.3)--(2.5,1.3);
			\fill[color = red, opacity = 0.5]  (2.5,1.35)--(2.1,1.3)--(2.5,1.3);
			\draw[color = red, latex-latex] (2.5,1.05)--(2.1,1)--(2.5,0.95);
			\fill[color = red, opacity = 0.5]  (2.5,1.05)--(2.1,1)--(2.5,0.95);	
			\draw[color = red, -latex] (2.1,0.7)--(2.5,0.7);
			\fill[color = red, opacity = 0.5]  (2.5,0.65)--(2.1,0.7)--(2.5,0.7);
			\draw[color = red, -latex] (2.1,0.45)--(2.5,0.45);
			\fill[color = red, opacity = 0.5]  (2.5,0.4)--(2.1,0.45)--(2.5,0.45);
			\draw[color = red, latex-latex] (1,0.4)--(1.5,0.4);
			\fill[color = red, opacity = 0.5]  (1,0.4)--(1.5,0.4)--(1.5,0.25)--(1,0.25);	
			\draw[color = red, -latex] (1.8,2.1)--(1.8,2.5);
			\fill[color = red, opacity = 0.5]  (1.85,2.5)--(1.8,2.1)--(1.8,2.5);
			\draw[color = red, -latex] (1.5,2.1)--(1.5,2.5);
			\fill[color = red, opacity = 0.5]  (1.55,2.5)--(1.5,2.1)--(1.5,2.5);
			\draw[color = red, -latex] (0.45,2.1)--(0.45,2.5);
			\fill[color = red, opacity = 0.5]  (0.4,2.5)--(0.45,2.1)--(0.45,2.5);
			\draw[color = red, latex-latex] (0.95,2.5)--(1,2.1)--(1.05,2.5);
			\fill[color = red, opacity = 0.5]  (0.95,2.5)--(1,2.1)--(1.05,2.5);
			\draw[color = red, latex-latex] (0.4,2)--(0.4,1.5);
			\fill[color = red, opacity = 0.5]  (0.4,2)--(0.4,1.5)--(0.3,1.5)--(0.3,2);
			\draw[color = red, latex-latex] (0.3,0.9)--(0.3,1.1);
			\fill[color = red, opacity = 0.5]  (0.2,0.9)--(0.2,1.1)--(0.3,1.1)--(0.3,0.9);
			\draw[color = red, latex-latex] (0.15,0.5)--(0.15,0.8);
			\fill[color = red, opacity = 0.5]  (0.15,0.5)--(0.15,0.8)--(0.05,0.8)--(0.05,0.5);
			\draw[color = red, -latex] (0.15,0.4)--(0.15,0.2);
			\fill[color = red, opacity = 0.5]  (0.15,0.4)--(0.15,0.2)--(0.1,0.2);
			\draw[color = red, -latex] (0.4,0.15)--(0.2,0.15);
			\fill[color = red, opacity = 0.5]  (0.4,0.15)--(0.2,0.15)--(0.2,0.1);
			\draw[color = red, latex-latex] (0.3,0.5)--(0.5,0.5)--(0.5,0.3);
			\fill[color = red, opacity = 0.5]  (0.3,0.5)--(0.5,0.5)--(0.5,0.3);
			\end{tikzpicture}
		\caption{Cones $B^{\delta}(N_i)$ for a representative fan. Notice that if the relative interior of $N_i$ contains the half-line $y=x<1$, then $B^{\delta}(N_i) = - \bR^2_{> 0}$ does not depend on $\delta$ and if the relative interior of $N_i$ contains the half-line $y=x>1$, then $B^{\delta}(N_i) = \bR^2_{> 0}$ does not depend on $\delta$. Likewise, if $N_i$ is adjacent to the line $x =0$ or the line $y=0$, then $B^{\delta}(N_i)$ is a half-plane that does not depend on $\delta$. The dependence on limiting tangents of curves of the form of $C(t)$ inspired the adjective ``tropical"\cite{macsturm}.}\label{escapedirsfig}
\end{center}
\end{figure}
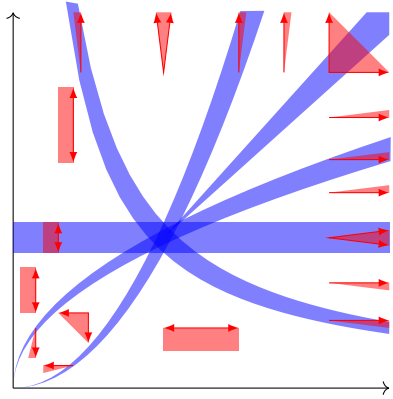

\subsection{Tropically Endotactic Differential Inclusions}

Now, we define the class of differential inclusions that we are most interested in.

\begin{definition}\label{tropendo}
A differential inclusion in $\bR^2_{>0}$ is \emph{tropically endotactic} if it is embedded in some $\cN$-cone differential inclusion $\cK(\cN)$ such that for every $N\in \cN\setminus \{\b{1}\}$, there is some $\delta>0$ such that
\[K(N) \cap B^{\delta}(N)^{\circ} = \emptyset\]
and if $N \in \cN$ is a face of $M \in \cN$, then $K(M) \subseteq K(N)$.
\end{definition}

%% file: main_proof_rev.tex
\begin{theorem}\label{thethm}
Let $\b{\dot{x}}\in \b{F}(\b{x})$ be a differential inclusion defined on $\bR^2_{> 0}$. If $\b{\dot{x}}\in \b{F}(\b{x})$ is tropically endotactic, then it is persistent and has bounded trajectories. Furthermore, there exists a nested family of compact regions which do not intersect $\partial \bR^2_{>0}$ which are forward invariant under $\b{\dot{x}}\in \b{F}(\b{x})$, and this family covers $\bR^2_{>0}$.
\end{theorem}

We will prove this theorem by constructing a forward invariant region with a polygonal border. We do this by choosing line segments that the tropically endotactic differential inclusion cannot cross, and connecting these segments to form the border of the forward invariant region. Finally, we show that there is an exhaustive nested family of such regions. This is achieved by ``expanding" the first region outward to cover $\bR^2_{>0}$.

In order to prove \cref{thethm}, we will first need to prove two lemmas. The first lemma allows us to pick curves through points which will be used as scaffolding for the construction of a forward invariant region, as in \cref{labeling} and \cref{labelingh2}. These curves will also give a framework to expand the forward invariant region and show there exists an exhaustive family of such regions.

\begin{lemma}\label{curves}
	Let $\widehat{N} \neq \b{0}$ be a cone of the fan $\widehat{\cN}$ in $\bR^2$ and $N = \exp(\widehat{N})$. For any point $\b{x}\in \mathit{fat}_{\fp}(N)$, there exists some curve
	\begin{equation}\label{curveform}
	S(t) = \{(t^{m_1},\alpha t^{m_2})\}\quad m_1,m_2 \in \bR,\,\alpha >0
	\end{equation}
	or 
	\begin{equation}\label{curveform2}
	S(t) = \{(\alpha t^{m_1}, t^{m_2})\}\quad m_1,m_2 \in \bR,\,\alpha >0
	\end{equation}
	such that for some $t_0>1$, we have $\{S(t)|t>t_0\}\subset \mathit{fat}_{\fp}(N)$ and $\b{x} = S(t_0)$.
\end{lemma}

\begin{proof}
	This lemma must be proved separately for one and two dimensional cones $\widehat{N}$.
	
	Let $N = \exp(\widehat{N})$ be one dimensional, and let $\b{m} =(m_1,m_2)$ be the direction of the ray $\widehat{N}$. We have that
	\[
	\mathit{fat}_{\fp}(\widehat{N})\subset \left\{\b{X}|dist(\b{X},\widehat{N})<|\log(\fp)|\right\}
	\]
	and any point $\b{X} \in \left\{\b{X}|dist(\b{X},\widehat{N})<|\log(\fp)|\right\}$ lies on a line $(m_1t,m_2t)+\beta[\b{m}^{\perp}]_n$, \\$|\beta|<|\log(\fp)|$. Furthermore, if $dist(\b{X},\widehat{M})>|\log(\fp)|$ for $M \neq N \in \cN$, $\|\b{Y}\|>\|\b{X}\|$ and $\b{Y}$ lies on the same line $(m_1t,m_2t)+\beta[\b{m}^{\perp}]_n$, then $dist(\b{Y},\widehat{M})$ increases with $\|\b{Y}\|$, so $dist(\b{Y},\widehat{M})>|\log(\fp)|$. Therefore, $\mathit{fat}_{\fp}(\widehat{N})$ is a union of affine half-lines. If $\b{x}\in \mathit{fat}_{\fp}(N)$ and $m_1 \neq 0$, then there is one such affine half-line $\widehat{S}(t)$ such that $\log(\b{x}) = \widehat{S}(t_0)$ and $S(t) = \exp(\widehat{S}(t))$ can be reparameterized to satisfy \cref{curveform}. If $m_1 = 0$, $\widehat{S}(t)$ can be chosen such that $S(t)$ can be reparameterized to satisfy \cref{curveform2}.
	
	Next, let $\widehat{N}$ be a solid cone, and let $\b{\mu}_1$ be the direction of one face and $\b{\mu}_2$ the direction of the other. Let $\widehat{M}^1$ and $\widehat{M}^2$ be the one dimensional faces of $\widehat{N}$, and $\b{X}$ be the unique intersection point $\mathit{fat}_{\fp}(\widehat{M}^1)\cap \mathit{fat}_{\fp}(\widehat{M}^2)$. Then
	\[
	\mathit{fat}_{\fp}(\widehat{N}) = \{\b{X} + t(\lambda \b{\mu}_1 + (1-\lambda)\b{\mu}_2)|\lambda\in [0,1]\}
	\]
	Therefore, if $\b{x}\in \mathit{fat}_{\fp}(N)$, then $\log(\b{x}) \in \b{X} + t(\lambda \b{\mu}_1 + (1-\lambda)\b{\mu}_2) = \b{X} + t \b{\mu} \eqqcolon \widehat{S}(t)$ for some $\lambda \in [0,1]$. Then again the curve $S(t) = \exp(\widehat{S}(t))$ satisfies the condition in the lemma.
\end{proof}


In the following lemma, we show that a line that intersects two curves of the form \cref{curveform} transversally will also intersect both of these curves at points further along (as $t$ increases), as in \cref{toftaufig}. Later, when we construct a forward invariant region for a tropically endotactic differential inclusion in the proof of \cref{thethm}, we will construct a set of lines that solutions of the differential inclusion cannot cross. This lemma will allow us to arrange these lines so that they form the border of a compact forward invariant polygon. Furthermore, this lemma will ensure that, once we have built one forward invariant polygon with sides along lines of the form $\{S(\tau_1) + s\b{v}|s \in \bR\}$, we can expand it into a continuous family of forward invariant polygonal regions  with sides along lines of the form $\{S(\tau^*) + s\b{v}|s \in \bR\}$, where $\tau^*>\tau_1$ (see \cref{toftaufig}).

\begin{lemma}\label{toftau}
	Let $S_1(\tau) = (\tau^{\tilde{m}_1},\tilde{\alpha} \tau^{\tilde{m}_2})$ and $S_2(t)=(t^{m_1},\alpha t^{m_2})$, $t,\tau>1$, such that $\mathit{sgn}(\max\{m_1,m_2\}) = \mathit{sgn}(\max\{\tilde{m}_1,\tilde{m}_2\}) \neq 0$.
	\footnote{
	Informally, $\mathit{sgn}(\max\{m_1,m_2\}) = \mathit{sgn}(\max\{\tilde{m}_1,\tilde{m}_2\})$ means that either both curves $S_1$ and $S_2$ or neither have logarithmic images in the third quadrant (with some exceptions along the boundary of the third quadrant).	
	}
	Assume there is some $\b{v} \in \bR^2$ such that 
	\[
	\b{n}(S_1(\tau)) \cdot \b{v} \neq 0,\;\forall \,\tau\geq \tau_0 \quad \& \quad \b{n}(S_2(t)) \cdot \b{v} \neq 0,\;\forall \,t\geq t_0
	\]
	where $\b{n}(S_i(t))$ is the normal to $S_i$ at $t$.\footnote{
	Note, $\b{n}(S_i(t)) \cdot \b{v} \neq 0,\;\forall t\geq t_0$ means that $\b{v}$ is not tangent to the curve $S_i(t)$ for $t\geq t_0$.
}
	Let $L_{\tau}(s) = S_1(\tau) + s\b{v}$, consider the half-line
	\[L_{\tau} = \{L_{\tau}(s)|s>0\} \]
	and assume that there is some $\tau_1\geq \tau_0$ such that $L_{\tau_1}\cap S_2(t) \neq \emptyset$, and let $S_2(t_1) = L_{\tau_1}(s_0)$. 
		
	Then, there exists a continuous and invertible function $\gamma(\tau)$ with $\gamma(\tau_1) = t_1$ such that $S_2(\gamma(\tau)) \in L_{\tau}$, and $\gamma(\tau)$ is strictly monotonically increasing. Furthermore, $\lim_{\tau\rightarrow \infty}\gamma(\tau)  = \infty$.
\end{lemma}

\begin{figure}
	\begin{center}
		\begin{tikzpicture}[scale = 0.6]
		\draw[->] (0,0) -- (0,7);
		\draw[->] (0,0) -- (7,0);
		\draw[color = blue, very thick] plot[domain = 7:0] (0.1*\x^2,\x);
		\draw[color = blue, very thick] plot[domain = 7:0] (\x,0.1*\x^2);
		\node[color = blue] at (7,5.3) {$S_1(\tau)$};
		\node[color = blue] at (5.6,7) {$S_2(t)$};
		\fill (5,2.5) circle (2pt) node[right] {$S_1(\tau_1)$};
		\fill (2.5,5) circle (2pt) node[right] {$S_2(t_1)$};
		\draw[-latex,thick] (5,2.5)--(4.5,3) node[pos = 1, above] {$\b{v}$};
		\draw[color = dgreen] (5,2.5)--(2.5,5);
		\draw[dashed, color = dgreen] (5.5,2)--(2,5.5) node[pos = 1,above] {$L_{\tau_1}$};
		\fill (4,1.6) circle (2pt) node[right] {$S_1(\tau^*)$};
		\fill (1.6,4) circle (2pt) node[right] {$S_2(\gamma(\tau^*))$};
		\draw[-latex,thick] (4,1.6)--(3.5,2.1) node[pos = 1, above] {$\b{v}$};
		\draw[color = dgreen] (4,1.6)--(1.6,4);
		\draw[dashed, color = dgreen] (4.5,1.1)--(1.1,4.5) node[pos = 0,below] {$L_{\tau^{*}}$};
		\draw[-latex,thick, color = nicepink] (3,0.9)--(2.7,1.4) node[pos = 0.5, above left] {$\b{n}(S_{1}(\tau))$};
		\draw[-latex,thick, color = nicepink] (0.9,3)--(0.4,3.3) node[pos = 1, left] {$\b{n}(S_{2}(t))$};
		\end{tikzpicture}
	\end{center}
	\caption{Given the segment $L_{\tau_1}$ intersecting two curves $S_1(\tau)$ and $S_2(t)$ transversally (as in the hypothesis of \cref{toftau}), this lemma insures that for $\tau^*>\tau_1$, the segment $L_{\tau^*}$, which is parallel to $L_{\tau_1}$, also intersects the two curves transversally at the points $S_1(\tau^*)$ and $S_2(t^*)$ respectively, where $t^*>t_1$. Notice that as $t^*,\tau^*$ increase, the segment moves closer to $\partial \bR^2_{>0}$.}\label{toftaufig}
\end{figure}
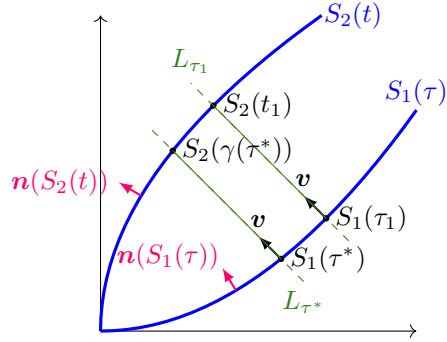

\begin{proof}
	If $S_2(t) \in L_{\tau}$, then
	\[
	(t^{m_1},\alpha t^{m_2})= (\tau^{\tilde{m}_1}+sv_1,\tilde{\alpha}\tau^{\tilde{m}_2}+sv_2)
	\] 
	for some $s$, or rewritten,
	\begin{align*}
	t^{m_1} - sv_1 & = \tau^{\tilde{m}_1}\\
	\alpha t^{m_2} - sv_2 &= \tilde{\alpha}\tau^{\tilde{m}_2}
	\end{align*}
	Because one of $v_1$, $v_2$ is non-zero, this is equivalent to 
	\begin{equation}
	g(\tau,t) = v_1(\alpha t^{m_2} - \tilde{\alpha}\tau^{\tilde{m}_2}) - v_2(t^{m_1} - \tau^{\tilde{m}_1}) = 0
	\end{equation}
	We know a solution exists at $(\tau_1,t_1)$. Furthermore,
	\[
	\frac{\partial g}{\partial t} = v_1\alpha m_2 t^{m_2-1} - v_2m_1t^{m_1-1} = \b{v} \cdot \left(\alpha m_2 t^{m_2-1}, - m_1t^{m_1-1}\right) = \b{v} \cdot \b{n}(S_2(t))
	\]
	and so $\frac{\partial g}{\partial t}(t) \neq 0$ for $t\geq t_1$. Then the implicit function theorem tells us there is a continuous, differentiable function $\gamma$ such that $g(\tau,\gamma(\tau)) = 0$ for $\tau \in [\tau_1,\tau^*)$ for some $\tau^*$. Likewise the implicit function theorem proves the existence of $\gamma^{-1}$ locally.
	
	Now, let $t = \gamma(\tau)$. Using the chain rule,
	\[
	\gamma'(\tau) =  \frac{-\nicefrac{\partial g}{\partial \tau}}{\nicefrac{\partial g}{\partial t}} = \frac{\b{v} \cdot \left(\tilde{\alpha} \tilde{m}_2 \tau^{\tilde{m}_2-1}, - \tilde{m}_1\tau^{\tilde{m}_1-1}\right)}{\b{v} \cdot \left(\alpha m_2 t^{m_2-1}, - m_1t^{m_1-1}\right)} = \frac{\b{v}\cdot \b{n}(S_1(\tau))}{\b{v} \cdot \b{n}(S_2(t))}
	\]
	 where both normal vectors are on the right-hand side of their respective curve with respect to the positive tangent direction. 
	 Notice that together, these curves form the boundary of a simply connected region in $\bR^2_{>0}$. The vector $\b{n}(S_1(t))$ is an inward normal to this region, while $\b{n}(S_2(t))$ is an outward normal (because these two curves have opposite orientation along the boundary of this region). The line $L_{\tau} = \{S_1(\tau) + s\b{v}|s>0\}$ enters this region across $S_1(t)$ and exits across $S_2(t)$. This gives us the fact that
	 \[
	 \b{n}(S_1(\tau)) \cdot \b{v} > 0,\;\forall \,\tau\geq \tau_0 \quad \& \quad \b{n}(S_2(t)) \cdot \b{v} > 0,\;\forall \,t\geq t_0
	 \] 
	 Therefore, $\gamma'(\tau) > 0$ for $\tau>\tau_1$.
	 
	 To show that $\gamma$ exists for $[\tau_1,\infty)$, we must show that if $\gamma$ exists for $[\tau_1,\tau^*)$, then  $\gamma(\tau^*)$ exists and is finite. To do this, it is sufficient to show that $\gamma(\tau)$  is bounded in $[\tau_1,\tau^*)$.
	 On the interval $[\tau_1,\tau^*)$, we have that
	 \[
	 v_1\alpha \gamma(\tau)^{m_2} - v_2 \gamma(\tau)^{m_1} = v_1 \tilde{\alpha}\tau^{\tilde{m}_2} - v_2 \tau^{\tilde{m}_1}
	 \]
	 We must now treat two different cases.
	 
	 If $\mathit{sgn}(\max\{m_1,m_2\}) = \mathit{sgn}(\max\{\tilde{m}_1,\tilde{m}_2\})>0$, we have some $M_1$ such that  
	 \[
	 \left| v_1 \tilde{\alpha}\tau^{\tilde{m}_2} - v_2 \tau^{\tilde{m}_1}\right| < M_1
	 \]
	 for $\tau \in [\tau_1,\tau^*]$. Therefore, we have that 
	 \[
	 \left| v_1\alpha \gamma(\tau)^{m_2} - v_2 \gamma(\tau)^{m_1} \right| < M_1
	 \]
	 for $t \in [\tau_1,\tau^*)$. If $m_1 \neq m_2$, then $\gamma(\tau)$ must be bounded on $[\tau_1,\tau^*)$ because if not, the difference $\left| v_1\alpha \gamma(\tau)^{m_2} - v_2 \gamma(\tau)^{m_1} \right|$ could not be bounded.
	 If $m_1 = m_2$, we must have $\alpha \neq \frac{v_2}{v_1}$. This is because if $m_1 = m_2$, then $\b{n}(S_1(\tau))$ is parallel to $(\alpha,-1)$. Therefore, $\alpha = \frac{v_2}{v_1}$ would violate the assumption that $\b{v} \cdot \b{n}(S_1(\tau))  \neq 0$. Therefore $| v_1\alpha \gamma(\tau)^{m_2} - v_2 \gamma(\tau)^{m_1}| = |\gamma(\tau)^{m_2}(v_1\alpha - v_2)| < M_1$ again implies that $\gamma(\tau)$ is bounded in $[\tau_1,\tau^*)$.
	 
\medskip
	 
	 If $\mathit{sgn}(\max\{m_1,m_2\}) = \mathit{sgn}(\max\{\tilde{m}_1,\tilde{m}_2\}) < 0$, there are three sub-cases to consider. Again, we have that 
	 \[
	 v_1\alpha \gamma(\tau)^{m_2} - v_2 \gamma(\tau)^{m_1} = v_1 \tilde{\alpha}\tau^{\tilde{m}_2} - v_2 \tau^{\tilde{m}_1}
	 \]
	 and we must in all cases show that $v_1 \tilde{\alpha}\tau^{\tilde{m}_2} - v_2 \tau^{\tilde{m}_1}$ cannot be $0$ on the compact interval $[\tau_1,\tau^*]$. 
	 
	 If $\tilde{m}_1 = \tilde{m}_2$, we again obtain $\tilde{\alpha} \neq \frac{v_2}{v_1}$ and can conclude $v_1 \tilde{\alpha}\tau^{\tilde{m}_2} - v_2 \tau^{\tilde{m}_1} \neq 0$ for $\tau \in [\tau_1,\tau^*]$.
	 
	 If $0>\tilde{m}_1 > \tilde{m}_2$ then we recall
	 \[
		0 <  \b{n}(S_1(\tau)) \cdot \b{v} = v_1\tilde{\alpha}\tilde{m}_2\tau^{\tilde{m}_2} - v_2 \tilde{m}_1\tau^{\tilde{m}_1}
	 \]
	 which implies that
	 \[
	 v_2\tau^{\tilde{m}_1} > v_1\tilde{\alpha}\frac{\tilde{m}_2}{\tilde{m}_1} > v_1\tilde{\alpha}\tau^{\tilde{m}_2}
	 \]
	 and so 
	 \[v_1 \tilde{\alpha}\tau^{\tilde{m}_2} - v_2 \tau^{\tilde{m}_1} \neq 0\]
	 for $\tau \in [\tau_1,\tau^*]$. 
	 
	 The last case, that $0> \tilde{m}_2 > \tilde{m}_1$ is analogous. Finally, we have that there is some $\epsilon$ such that
	 \[
	 |v_1\alpha \gamma(\tau)^{m_2} - v_2 \gamma(\tau)^{m_1}| > \epsilon
	 \]
	 for $\tau \in [\tau_1,\tau^*)$. Once more, if $m_1 = m_2$, then $\alpha \neq \frac{v_2}{v_1}$, and we can conclude that $\gamma(\tau)$ is bounded in $[\tau_1,\tau^*)$. This, along with monotonicity of $\gamma(\tau)$, allows us to conclude that $\lim_{\tau\rightarrow \tau^*}\gamma(\tau)=\gamma(\tau^*)$ exists and is finite. Therefore, whenever $\gamma(\tau)$ can be defined on $[\tau_1,\tau^*)$, it can be extended to $[\tau_1,\tau^*]$.
	 
	 If we assume that $\gamma(\tau)$ cannot be defined on $[\tau_1,\infty)$, then there exists some interval $[\tau_1,\tau^{**})$ on which $\gamma(\tau)$ exists but cannot be extended to $\tau^{**}$. This is a contradiction, so we conclude that $\gamma(\tau)$ exists on $[\tau_1,\infty)$. Notice that the lemma is symmetric in $\tau$ and $t$, and so we have also that $\gamma(\tau) \rightarrow \infty$ as $\tau \rightarrow \infty$ (by repeating the above arguments for $\gamma^{-1}(t)$ to show that $\gamma^{-1}(t)$ exists on $[t_1,\infty)$).
\end{proof}

\bigskip

\begin{proof}[Proof of \Cref{thethm}]
Let $\cN$ be a complete fan, and $\cK(\cN)$ an $\cN$-cone differential inclusion satisfying the hypothesis of \cref{tropendo} such that $\b{\dot{x}} \in \b{F}(\b{x})$ is embedded in $\cK(\cN)$. We will construct a family of nested, forward invariant regions for $\cK(\cN)$. Clearly, any region which is forward invariant under $\cK(\cN)$ is forward invariant under any differential inclusion that is embedded in $\cK(\cN)$. 

The regions constructed will be compact, will not intersect $\partial \bR^2_{> 0}$, and will cover $\bR^2_{>0}$. We construct the border of one such region $\cR$ and show that a nested family exists, all other members of which properly contain $\cR$. The construction proceeds from one region $\mathit{fat}_{\fp}(N)$ to the next in clockwise manner about the point $\b{1} = (1,1)$, so we number the regions $N_i \in \cN$ clockwise (excluding $\mathit{fat}_{\fp}(\b{1})$ from this numbering).

We will build $\cR$ to be a (not necessarily convex) polygon together with its interior. Each side of the polygon will be chosen so that if $\b{n}_j$ is the inward normal to that side and the side intersects $\mathit{fat}_{\fp}(N_i)$, then $\b{w} \cdot \b{n}_j \geq 0$ for all $\b{w} \in K(N_i)$. That is, the sides will be chosen to be supporting lines of the cones $K(N_i)$. Each corner of the polygon $\partial \cR$ will be chosen to be in the interior of a region $\mathit{fat}_{\fp}(N_i)$, so any two intersecting edges of $\partial \cR$ are supporting lines of the single cone $K(N_i)$. An application of Theorem 5.2.7 of \cite{diffincbook} shows that $\cR$ constructed in this manner is forward invariant for $\cK(\cN)$.

According to \cref{curves}, we may choose the vertices of the polygon $\partial \cR$ on curves $S_i(t) = \{(t^{m_1},\alpha t^{m_2})\} \subseteq \mathit{fat}_{\fp}(N_i)$, $m_1,m_2 \in \bR$, $\alpha>0$, for $t$ large. In our construction, we will choose these curves first, and choose one point on each curve sequentially following line segments in the direction of supporting lines of the appropriate cones $K(N_i)$. These line segments then form the edges of $\partial \cR$. We number these curves clockwise with respect to a neighborhood of the point $\b{1}$ (some regions $\mathit{fat}_{\fp}(N_i)$ will have more than one curve, we then label these as $S_i'$ and $S_i''$, see \cref{labeling} and \cref{labelingh2}).

If we can connect two curves with a segment in some direction, and these curves satisfy the conditions given by \cref{toftau}, we can also connect them with a segment in that direction which intersects both curves at points closer to $\partial\bR^2_{> 0}$. This allows us to adjust our segments so that adjacent segments intersect a curve at the same point. This will also allow us to show that there is a family of forward invariant regions that cover $\bR^2_{>0}$.

We must determine  $\b{n}_j$, the inward normal of $\partial \cR$ along edge $j$, using only local information in order to distinguish inward and outward directions before completing the construction of $\cR$. To do this, we specify $\b{v}_j$ to be the clockwise direction of an edge, defined to mean that $\b{v}_j$ points from $S_i$ to $S_{i+1}$ (or $S_i'$ to $S_i''$, see \cref{labeling} and \cref{labelingh2}). In this way we have specified an orientation for the edges of $\partial \cR$. Then, for $\b{x}$ along this segment, $\b{n}(\b{x}) = \b{n}_j$ is the unit vector perpendicular to $\b{v}_j$ such that if we take the determinant
\[
det(\b{n}_j ,\b{v}_j) > 0
\]
We will call $\b{n}_j$ the clockwise normal to the segment with direction $\b{v}_j$, and construct $\cR$ so that the inward normal to $\partial \cR$ along an edge is the clockwise normal to that edge.

We will construct $\partial \cR$ in four parts. Two of these will be segments of the lines $\{x=x^*\}$, $\{y=y^*\}$ where $x^* < 1$ and $y^*<1$ are chosen so that they do not intersect $\mathit{fat}_{\fp}(\b{1})$. Note that these must be supporting directions of $K(N_i)$ in regions $\mathit{fat}_{\fp}(N_i)$ such that $\{X=0,Y<0\}\subseteq \widehat{N}_{i}$ and  $\{X<0,Y=0\}\subseteq \widehat{N}_{i}$ respectively, as seen by calculating $B^{\delta}(N_i)$ for such regions. To draw $\partial \cR$ through other regions, we can connect points on these lines with the other two parts of $\partial \cR$, which will be polygonal lines $H_1$ and $H_2$, as in \cref{sketch}. 

Throughout the proof, we will use $\{\b{\hat{x}},\b{\hat{y}}\}$ to denote the standard basis in $\bR^2$.

\begin{figure}
\begin{center}
\begin{tikzpicture}[scale = 0.6]
	\draw[->] (-0.5,0) -- (5,0);
    \draw[->] (0,-0.5) -- (0,5);
    \draw[thick] (0.1,0) -- (0.1,5) node[pos = 0.5, right]{$\{x  = x^*\}$};
    \draw[thick] (0,0.1) -- (5,0.1) node[pos = 1, right]{$\{y  = y^*\}$};
    \draw[dashed,thick] (2,0.1) -- (1.2,0.4)--(0.5,1.2)node[pos  = 0.5,right]{$H_1$} -- (0.1,2) ;
    \draw[dotted,very thick] (4, 0.1) -- (4.5,0.3)--(4.8,0.9)--(5.3,3.5)--(4.9,5)node[pos  = 0.5,left]{$H_2$}--(3.2,5.4)--(1.1,5)--(0.6,4.7)--(0.1,3.5);
\end{tikzpicture}
\caption{Sketch of how $\cR$ will be constructed. The four parts of the construction are: a segment of $x = x^*$, a segment of $y = y^*$, the polygonal line $H_1$, and the polygonal line $H_2$.}
\label{sketch}
\end{center}
\end{figure}
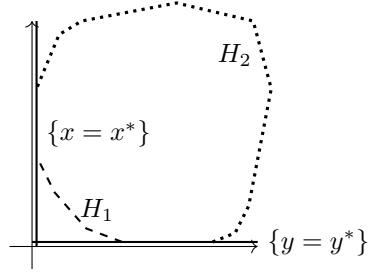
\bigskip
{\bf Construction of polygonal line $H_1$.}

Notice first that if the third quadrant is contained in some $\widehat{N}$, we connect the lines $x=x^*$ and $y=y^*$ and we can take $H_1$ to be their intersection point. Otherwise, we number regions as in \cref{labeling} (a). In this numbering, we have the segment $0<x=y<c\subset \mathit{fat}_{\fp}(N_l)$ for some $c\in (0,1)$. Next we choose curves for each $N_i$ with the form
\[
S(t) = \{(t^{m_1},\alpha t^{m_2})\} \quad m_1,m_2 < 0
\] 
so that there is some $t_0$ where $S_i(t) \subset \mathit{fat}_{\fp}(N_i)$ for $t>t_0$. These curves are numbered as in \cref{labeling} (b). Notice \emph{two} such curves are chosen inside $\mathit{fat}_{\fp}(N_l)$. We take, if possible, $m_1 \neq m_2$ for the curves $S_{l}'$ and $S_{l}''$. If this choice is not allowed, because such curves are not contained in $\mathit{fat}_{\fp}(N_l)$, we take $\alpha < 1$ for $S_{l}'$ and $\alpha>1$ for $S_l''$.

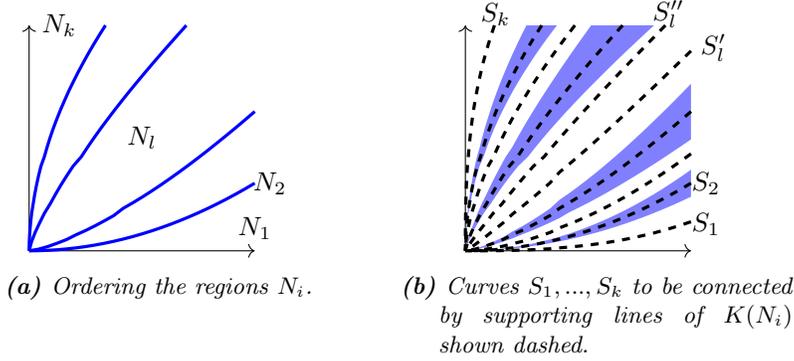
\begin{figure}
\begin{center}
\begin{subfigure}[t]{0.4\textwidth}
\begin{center}
\begin{tikzpicture}[scale = 10]
	\draw[->] (0,0) -- (0,0.3);
	\draw[->] (0,0) -- (0.3,0);
	\draw[color = blue,very thick] plot[domain = 0.3:0] (\x,\x^1.4);
	\draw[color = blue,very thick] plot[domain = 0.3:0] (\x^1.3,\x);
	\draw[color = blue,very thick] plot[domain = 0.3:0] (\x,\x^2);
	\draw[color = blue,very thick] plot[domain = 0.3:0] (\x^1.9,\x);
	\node at (0.3,0.03) {$N_1$};
	\node at (0.04,0.3) {$N_k$};
	\node at (0.15,0.15) {$N_{l}$};
	\node at (0.32,0.09) {$N_2$};
\end{tikzpicture}
\end{center}
\caption{Ordering the regions $N_i$.}
\end{subfigure}
\hspace{0.5 cm}
\begin{subfigure}[t]{0.4\textwidth}
\begin{center}
\begin{tikzpicture}[scale = 10]
	\draw[->] (0,0) -- (0,0.3);
	\draw[->] (0,0) -- (0.3,0);
	\path[fill,color = blue,opacity = 0.5] plot[domain = 0.3:0] (\x,1.2*\x^1.4) -- plot[domain = 0:0.3] (\x,0.8*\x^1.4);
	\path[fill,color = blue,opacity = 0.5] plot[domain = 0.3:0] (1.2*\x^1.3,\x) -- plot[domain = 0:0.3] (0.8*\x^1.3,\x);
	\path[fill,color = blue,opacity = 0.5] plot[domain = 0.3:0] (\x,1.2*\x^2) -- plot[domain = 0:0.3] (\x,0.8*\x^2);
	\path[fill,color = blue,opacity = 0.5] plot[domain = 0.3:0] (1.2*\x^1.9,\x) -- plot[domain = 0:0.3] (0.8*\x^1.9,\x);
	\draw[dashed, very thick] plot[domain = 0.3:0] (\x,\x^1.4);
	\draw[dashed, very thick] plot[domain = 0.3:0] (\x^1.3,\x);
	\draw[dashed, very thick] plot[domain = 0.3:0] (\x^1.9,\x);
	\draw[dashed, very thick] plot[domain = 0.3:0] (\x,\x^2);
	\draw[dashed, very thick] plot[domain = 0.3:0] (\x,\x^1.1);
	\draw[dashed, very thick] plot[domain = 0.3:0] (\x^1.1,\x);
	\draw[dashed, very thick] plot[domain = 0.3:0] (\x^1.6,\x);
	\draw[dashed, very thick] plot[domain = 0.3:0] (\x,\x^1.7);
	\draw[dashed, very thick] plot[domain = 0.3:0] (\x^2.7,\x);
	\draw[dashed, very thick] plot[domain = 0.3:0] (\x,\x^2.7);
	\node at (0.32,0.035) {$S_1$};
	\node at (0.32,0.09) {$S_2$};
	\node at (0.33,0.27) {$S_{l}'$};
	\node at (0.27,0.315) {$S_{l}''$};
	\node at (0.04,0.315) {$S_k$};
\end{tikzpicture}
\end{center}
\caption{Curves $S_1,...,S_k$ to be connected by supporting lines of $K(N_i)$ shown dashed.}
\end{subfigure}
\caption{Labeling used in the construction of $H_1$. The qualitative difference in limiting tangent between $S_1,...,S_{l-1},S_{l}'$ and $S_{l}'',S_{l+1},...,S_k$ (with the possibility that $S_{l}'$ and $S_{l}''$ are straight lines) makes it necessary to treat three different cases.}\label{labeling}
\end{center}
\end{figure}

We need to prove the following claim for each pair of adjacent curves in \cref{labeling} (b) (including one or both of the curves $S_l'$ and $S_l''$). This claim states that for any pair of adjacent curves, there exists a choice of direction vector $\b{v}$ for a connecting line segment which satisfies the hypothesis of \cref{toftau} and can serve as the boundary of an invariant region for the differential inclusion $\cK(\cN)$ in the region between these two curves.

\begin{clm}\label{cl1}
	Let $S_{-}(t)$ and $S_{+}(t)$, with $S_+$ clockwise of $S_-$, be adjacent curves as constructed above and $\widehat{N}_q$ be the one dimensional ray such that one of $S_{-}$ and $S_{+}$ is contained in $\mathit{fat}_{\fp}(N_q)$, or $N_q = N_i$ if $S_-,S_+ \subset N_i$. Denote by $\b{n}(S(t))$ a clockwise normal vector to the curve $S$ at the point $S(t)$. Then there exists a direction $\b{v}$ with clockwise normal $\b{n}_{\b{v}}$ such that
	\begin{enumerate}[(a)]
		\item for any $\b{w} \in K(N_q)$, we have $\b{w} \cdot \b{n}_{\b{v}} \geq 0$
		\item there exists $t_1>t_0$ such that $\b{n}(S_-(t)) \cdot \b{v} \neq 0$ for $t>t_1$
		\item there exists $t_2>t_1$ such that the line $L_{t_2} = S_-(t_2) + s\b{v}$ intersects $S_{-}$ at a point $\b{x} = S_{+}(\tau_2)$ where $\tau_2>t_0$
		\item $\b{n}(S_{+}(\tau))\cdot \b{v} \neq 0$ for $\tau\geq \tau_2$
	\end{enumerate}
\end{clm}

\begin{figure}[ht]
	\begin{center}
		\begin{tikzpicture}[scale = 1.6]
			\draw[->] (0,0) -- (0,2.5);
			\draw[->] (0,0) -- (2.5,0);
			\draw[dashed, very thick] plot[domain = 0:0.5] (5*\x,5*\x^2.5);
			 \node at (2.8,5*0.5^2.5) {$S_{-}(t)$};
			\draw[dashed, very thick] plot[domain = 0:0.5] (5*\x,5*\x^1.3);
			 \node at (2.8, 5*0.5^1.3) {$S_{+}(\tau)$};
			\draw[-latex,thick] (5*0.25,5*0.25^2.5) -- (5*0.25+0.13,5*0.25^2.5+0.4) node[left] {$\b{v}$};
			\draw[color = dgreen,thick] (5*0.25,5*0.25^2.5) -- (5*0.25+0.2925,5*0.25^2.5+0.9) node[pos = 0.2,right] {$L_t$};
			\draw[latex-latex,color = nicered,thick] (1.8,0.9) -- (1.445,0.75625) --(1.8,0.6) node[right,pos = 0.5] {$K(N_q)$}; 
			\draw[-latex, color = nicepink] (5*0.15,5*0.15^2.5)--(5*0.15-0.05,5*0.15^2.5+0.2) node[left] {$\b{n}(S_-(t))$};			\draw[-latex, color = nicepink] (5*0.2,5*0.2^1.3)--(5*0.2-0.15,5*0.2^1.3+0.2) node[above] {$\b{n}(S_{+}(\tau))$};
		\end{tikzpicture}
		\caption{The line $L_t$ intersects the curves $S_-(t)$ and $S_{+}(\tau)$.}
		\label{h1claim}
	\end{center}
\end{figure}
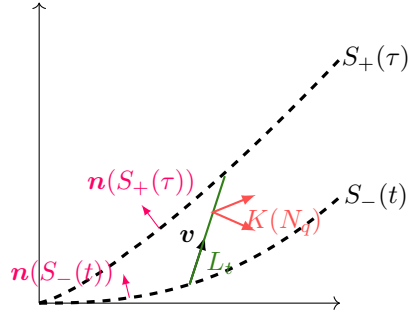

\emph{Proof of \cref{cl1}}. We need to consider three cases: $S_{-}\in\{S_1,...,S_{l-1}\}$,  $S_{-} = S_l'$, and $S_{-}\in\{S_{l}'',...,S_k\}$, with the first only if $l>2$.

Assume that $S_- = S_i \in \{S_1,...,S_{l-1}\}$ and let $q \in \{i,i+1\}$ be such that $N_q$ is one dimensional. Let $\b{v}$ be contained in a line that separates the cones $K(N_q)$ and $B^{\delta}(N_q)$, chosen such that $\b{v}\cdot \b{\hat{y}} >0$ or $\b{v} = -\b{\hat{x}}$. Such a choice is always possible because every line through the origin in $\bR^2$ either intersects the open upper half plane or is the $x$-axis. Then, we see that $\b{v}$ satisfies \cref{cl1} by checking the requirements in order.
\begin{enumerate}[(a)]
	\item There exists some vector $\b{u}$ arbitrarily close to $-\b{\hat{x}}$ in $B^{\delta}(N_q)^{\circ}$, and our choice of $\b{v}$ is such that $\det(\b{u},\b{v}) < 0$. Then, because $\b{v}$ is a direction that separates $B^{\delta}(N_q)$ and $K(N_q)$, we can conclude for $\b{w} \in K(N_q)$ that $\det(\b{w},\b{v})\geq 0$ which implies that $\b{w} \cdot \b{n}_{\b{v}} \geq 0$.
	
	\item We have that $\b{n}(S_i(t)) \to \b{\hat{y}}$, $\b{n}(S_i(t))\cdot (-\b{\hat{x}})\to 0$ monotonically, and $\b{n}(S_i(t))\cdot (-\b{\hat{x}}) \neq 0$ for all finite $t>t_0$. There is then $t_1>t_0$ such that $\b{n}(S_i(t)) \cdot \b{v} \neq 0$ for $t>t_1$.
	
	\item Both curves $S_i$ and $S_{i+1}$ approach horizontal and converge to the origin, and the curve $S_{i+1}$ is above the curve $S_{i}$. For any direction $\b{v}$ with $\b{v}\cdot \b{\hat{y}}>0$, there is some $t_2>t_1$ such that $L_{t_2}(s) = S_i(t_2) + s\b{v}_i$ intersects $S_{i+1}$ at $S_{i+1}(\tau_2)$, $\tau_2>\tau_1$. If $\b{v} = -\b{\hat{x}}$, we again have this intersection because $S_{i+1}\to 0$. 
	
	\item As in (b), We have that $\b{n}(S_{i+1}(t)) \to \b{\hat{y}}$, $\b{n}(S_{i+1}(t))\cdot -\b{\hat{x}}\to 0$ monotonically, and $\b{n}(S_{i+1}(t))\cdot (-\b{\hat{x}}) > 0$ for all finite $t>t_0$.
\end{enumerate}

If $S_- = S_{l-1}$ and $S_+ = S_{l}'$ has limiting tangent $-\b{\hat{x}}$, the preceding argument is valid. However, we may have that $S_{l}'$ is a line of positive slope. Then, $B^{\delta}(N_q) = -\bR^2_{>0}$ and we can choose $\b{v}$ to be the direction of a separating subspace which has $\b{v}\cdot \b{\hat{y}} \geq 0$ and $\b{v} \cdot (-\b{\hat{x}})\geq 0$. Then (a) and (b) are as in the previous case and we have that
\begin{enumerate}[(a)]
	\setcounter{enumi}{2}
	\item Here, $S_{l}'$ is a line of positive slope, and $S_{l-1}$ is below this line. Also, $\b{v} \cdot \b{\hat{y}}\geq 0$ and $\b{v} \cdot (-\b{\hat{x}})\geq 0$, so we will have the intersection desired.
	
	\item Because $S_{l}'$ is a line of positive slope, $\b{n}(S_l') \cdot \b{\hat{y}}\geq 0$ and $\b{n}(S_l') \cdot (-\b{\hat{x}})\geq 0$. The same is true of $\b{v}$, so (d) is easily satisfied.
\end{enumerate}

Next, we need to prove the claim for $S_- = S_{l}'$ and $S_+ = S_{l}''$. Again, $B^{\delta}(N_l) = -\bR^2_{>0}$ and we can choose $\b{v}$ to be the direction of a separating subspace which has $\b{v}\cdot \b{\hat{y}} \geq 0$ and $\b{v} \cdot -\b{\hat{x}}\geq 0$. Then, (a) is as in the first cases, and if $S_{l}'$ and $S_{l}''$ are lines of positive slope, (b),(c),(d) are as in the previous case. If not, (b) is as in the previous case, (d) is analogous, reflected across the line $y=x$ (and taking $-\b{v}$), and lastly:
\begin{enumerate}[(a)]
	\setcounter{enumi}{2}
	\item Note that $S_{l}''$ is above $S_{l}'$, while $\b{v}\cdot \b{\hat{y}}\geq 0$. Furthermore, $S_l''$ is to the left of $S_l'$, and $\b{v}\cdot -\b{\hat{x}}\leq 0$. A segment starting at $S_{l}'$ in the direction of $\b{v}$ will not intersect $S_{l}'$ again, and so the desired intersection must occur.
\end{enumerate}

For $S_- \in \{ S_{l}'',...,S_k\}$, we can reflect the previous argument across the line $y = x$. Reflection gives a counterclockwise construction of each $L_t$, with directions that satisfy $\b{v}\cdot \b{\hat{x}} >0$ or $\b{v} = -\b{\hat{y}}$. A continuing clockwise construction would instead use $-\b{v}$. This concludes the proof of \cref{cl1}.

\bigskip

\Cref{cl1} gives us a way to construct line segments that serve as a boundary of an invariant region for $\cK(\cN)$ within the individual regions between curves $S_{-}(t)$ and $S_{+}(t)$. Note that two such segments in adjacent regions do not necessarily intersect. On the other hand, results (b),(c), and (d) of the claim are three of the four conditions of \cref{toftau}. Furthermore, to construct $H_1$ we only considered curves whose logarithmic image is contained in the third quadrant, and so the last condition of \cref{toftau} is also satisfied. Therefore, by applying \cref{toftau} we can ``slide" each segment so that segments in adjacent regions meet at a point. Therefore, a solution of the differential inclusion $\cK(\cN)$ cannot cross $H_1$ in the outward direction.

\bigskip
{\bf Construction of polygonal line $H_2$.}

We number regions as in \cref{labelingh2} (a). In this numbering, we have the segment $c<x=y \subset \mathit{fat}_{\fp}(N_p)$, the segment $\{x = 1, y>c\} \subset \mathit{fat}_{\fp}(N_o)$, and the segment $\{x > c, y = 1\} \subset \mathit{fat}_{\fp}(N_r)$ for some $c \in (1,\infty)$. Next we choose curves with the form
\[
S(t) = \{(t^{m_1},\alpha t^{m_2})\} 
\] 
where either $m_1>0$ or $m_2 > 0$ so that there is some $t_0$ where $S_i(t) \subset \mathit{fat}_{\fp}(N_i)$ for $t>t_0$. These curves are numbered as in \cref{labelingh2} (b). Notice \emph{two} such curves are chosen inside each of $\mathit{fat}_{\fp}(N_o)$, $\mathit{fat}_{\fp}(N_p)$, and $\mathit{fat}_{\fp}(N_r)$. We choose for $S_{o}'$ that $m_1 \leq 0$ and $S_{o}''$ that $m_1 \geq 0$ with strict inequality if possible. Similarly, we choose for $S_{r}'$ that $m_2 \geq 0$ and for $S_{r}'$ that $m_2 \leq 0$ with strict inequality if possible. We choose for $S_{{p}}'$, $m_1\leq m_2$ and $\alpha>1$ if $m_1 = m_2$, and for $S_{p}''$, $m_2\leq m_1$ and $\alpha<1$ if $m_1 = m_2$. We choose $m_1 \neq m_2$ for these curves if this is a possible choice. 

\begin{figure}
\begin{center}
\begin{subfigure}[t]{0.45\textwidth}
\begin{center}
\begin{tikzpicture}[scale = 0.6]
	\draw[->] (0,0) -- (0,5);
	\draw[->] (0,0) -- (5,0);
	\draw[color = blue, very thick] plot[domain = 0:2.2] (\x^1.9,\x);
	\draw[color = blue, very thick] plot[domain = 0:2.2] (\x,\x^2);
	\draw[color = blue, very thick] plot[domain = 0.2:5] (\x,\x^-1);
	\node at (0.4,1.4) {$N_h$};
	\node at (4,4) {$N_p$};
	\node at (1.5,0.3) {$N_m$};
	\node at (1,4) {$N_o$};
	\node at (4,1) {$N_r$};
\end{tikzpicture}
\end{center}
\caption{Ordering the cones.}
\end{subfigure}
\hspace{0.5 cm}
\begin{subfigure}[t]{0.45\textwidth}
\begin{center}
\begin{tikzpicture}[scale = 0.6]
	\draw[->] (0,0) -- (0,5);
	\draw[->] (0,0) -- (5,0);
	\path[fill,color = blue,opacity = 0.5] plot[domain =2.03:0] (\x,1.2*\x^2) -- plot[domain = 0:2.5] (\x,0.8*\x^2);
	\path[fill,color = blue,opacity = 0.5] plot[domain =2.17:0] (1.2*\x^1.9,\x) -- plot[domain = 0:2.7] (0.8*\x^1.9,\x);
	\path[fill,color = blue,opacity = 0.5] plot[domain =0.16:5] (\x,0.8*\x^-1) -- plot[domain = 5:0.3] (\x,1.5*\x^-1);
	\draw[dashed, very thick] plot[domain = 1:2.8] (\x,\x^1.4);
	\draw[dashed, very thick] plot[domain = 1:2.7] (\x^1.5,\x);
	\draw[dashed, very thick] plot[domain = 1:1.67] (\x,\x^2.9);
	\draw[dashed, very thick] plot[domain = 1:4.5] (\x^-0.3,\x);
	\draw[dashed, very thick] plot[domain = 1:1.67] (\x^2.9,\x);
	\draw[dashed, very thick] plot[domain = 1:2.4] (\x^1.9,\x);
	\draw[dashed, very thick] plot[domain = 1:2.25] (\x,\x^2);
	\draw[dashed, very thick] plot[domain = 0.2:5] (\x,\x^-1);
	\draw[dashed, very thick] plot[domain = 1:5] (\x,\x^-1.5);
	\draw[dashed, very thick] plot[domain = 0.04:1] (\x,\x^-0.5);
	\draw[dashed, very thick] plot[domain = 1:4.5] (\x,\x^-0.3);
	\node at (2.9,4.35) {$S_{p}'$};
	\node at (4.8,2.9) {$S_{p}''$};
	\node at (1.6,4.7) {$S_{o}''$};
	\node at (0.7,4.8) {$S_{o}'$};
	\node at (4.9,1.7) {$S_{r}'$};
	\node at (4.9,0.7) {$S_{r}''$};
	\node at (-0.3,4.5) {$S_{h}$};
	\node at (5.4,0.1) {$S_m$};
\end{tikzpicture}
\end{center}
\caption{Curves $S_h,S_{h+1},...,S_m$ to be connected by supporting lines of $K_{G}^{\delta}(D_i)$.}
\end{subfigure}
\caption{Labeling used in the construction of $H_2$. Notice that the logarithmic images of $S_{h},S_{h+1},...,S_{o}'$ are contained in the second quadrant, the logarithmic images of $S_{o}'',S_{o+1},...,S_{r}'$ are contained in the first quadrant, and the logarithmic images of $S_{r}'',S_{r+1},...,S_m$ are contained in the fourth quadrant. The qualitative difference in limiting tangent between the curves $S_h,S_{h+1},...,S_m$ makes it necessary to treat several different cases.}\label{labelingh2}
\end{center}
\end{figure}
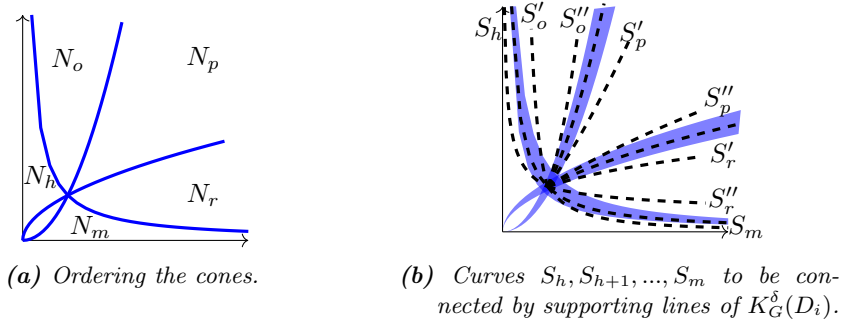

Just as in the construction of $H_1$, we need to prove a claim for each adjacent pair $S_-(t)$, $S_{+}(t)$ (with some pairs including one or both of $S_o'$,$S_o''$ and $S_r'$, $S_r''$ and $S_p'$, $S_p''$). As before, we will make the claim that for any pair of adjacent curves, there exists a choice of direction vector for a connecting line segment which satisfies the hypothesis of \cref{toftau} and can serve as the boundary of an invariant region for the differential inclusion $\cK(\cN)$ in the region between these two curves. The difference between the claim below and \cref{cl1} is that we are now dealing with curves whose logarithmic images are outside of the third quadrant.

\begin{clm}\label{cl2}
	Let $S_-(t)$ and $S_{+}(t)$, with $S_+$ clockwise of $S_-$, be adjacent curves as constructed above and let $\widehat{N}_q$ be the one dimensional ray such that one of $S_{-}$ and $S_{+}$ is contained in $\mathit{fat}_{\fp}(N_q)$, or $N_q = N_i$ if $S_-,S_+ \subset N_i$. Denote by $\b{n}(S(t))$ a clockwise normal vector to the curve $S$ at the point $S(t)$. Then there exists a direction $\b{v}$ with clockwise normal $\b{n}_{\b{v}}$ such that
	\begin{enumerate}[(a)]
		\item for any $\b{w} \in K(N_q)$, we have $\b{w} \cdot \b{n}_{\b{v}} \geq 0$
		\item there exists $t_1>t_0$ such that $\b{n}(S(t)) \cdot \b{v} \neq 0$ for $t>t_1$
		\item there exists $t_2>t_1$ such that the line $L_{t_2} = S_-(t_2) + s\b{v}$ intersects $S_{+}$ at a point $\b{x} = S_{+}(\tau_2)$ where $\tau_2>t_0$
		\item $\b{n}(S_{+}(\tau))\cdot \b{v} \neq 0$ for $\tau\geq \tau_2$
	\end{enumerate}
\end{clm}

\emph{Proof of \cref{cl2}}. As in the proof of \cref{cl1}, we have more than one case to consider.

First, assume that $S_- = S_i\in \{S_h,...,S_{o-1}\}$. Let $\b{v}$ be contained in a line that separates the cones $K(N_q)$ and $B^{\delta}(N_q)$, chosen such that $\b{v} \cdot \hat{\b{x}}>0$ if possible, and otherwise $\b{v} = \b{\hat{y}}$. Then, we have
\begin{enumerate}[(a)]
	\item There exists some vector $\b{u}$ arbitrarily close to $\b{\hat{y}}$ in $B^{\delta}(N_q)^{\circ}$ with $\det(\b{u},\b{\hat{y}})<0$. Then, our choice of $\b{v}$ is such that $\det(\b{u},\b{v}) < 0$. Then, because $\b{v}$ is a direction that separates $B^{\delta}(N_q)$ and $K(N_q)$, we can conclude for $\b{w} \in K(N_q)$ that $\det(\b{w},\b{v})\geq 0$ which implies that $\b{w} \cdot \b{n}_{\b{v}} \geq 0$.
	\item We have that $\b{n}(S_i(t)) \to \b{\hat{x}}$, $\b{n}(S_i(t))\cdot (\b{\hat{y}})\to 0$ monotonically, and $\b{n}(S_i(t))\cdot (\b{\hat{y}}) \neq 0$ for all finite $t>1$. There is then $t_1$ such that there exists $t_1>t_0$ such that $\b{n}(S_i(t)) \cdot \b{v} \neq 0$ for $t>t_1$.
	\item Both $S_{i}$ and $S_{i+1}$ approach vertical, and $S_{i}$ is to the left of $S_{i+1}$. Therefore, if $\b{v} \cdot \b{\hat{x}} > 0$, the desired intersection occurs. If $\b{v} = \b{\hat{y}}$, we have that $S_{i+1}$ approaches the line $x=0$ and is above the curve $S_{i}$. Then, the desired intersection must occur.
	\item If $i<o-1$ or $S_{o}'$ is not a line, this is as in (b). If $i=o-1$ and $S_{o}'$ is a line, it is vertical and $\b{v} \neq \hat{\b{y}}$, because $\b{\hat{y}} \in B^{\delta}(N_o)^{\circ}$.
\end{enumerate}

The above holds as well if $S_- = S_o'$, again because $\b{\hat{y}} \in B^{\delta}(N_o)^{\circ}$.

The above holds also for $S_-\in \{S_o'',...,S_{p-2}\}$ unless we have that the only separating line between $K(N_q)$ and $B^{\delta}(N_q)$ is vertical. Then we take $\b{v} = -\b{\hat{y}}$ and (b), and (d) are the same, while
\begin{enumerate}[(a)]
	\item There exists some vector $\b{u}$ arbitrarily close to $\b{\hat{y}}$ in $B^{\delta}(N_q)^{\circ}$ with $\det(\b{u},-\b{\hat{y}})<0$. Then, because $\b{v}$ is a direction that separates $B^{\delta}(N_q)$ and $K(N_q)$, we can conclude for $\b{w} \in K(N_q)$ that $\det(\b{w},\b{v})\geq 0$ which implies that $\b{w} \cdot \b{n}_{\b{v}} \geq 0$.
	\setcounter{enumi}{2}
	\item  Both $S_-$ and $S_{+}$ are curves of the form $y = \alpha x^m$ with $m> 1$ with $S_{-}$ above $S_{+}$. Then, with $\b{v}_i = -\b{\hat{y}}$, we have the desired intersection.
\end{enumerate}

If $S_- = S_{p-1}$ and $S_{p}'$ has $m_1<m_2$, then the preceding argument is valid. If $S_{p}'$ is a straight line, then $N_p$ is a ray and we have that $B^{\delta}(N_{p}) = \bR^2_{>0}$. We can therefore choose a direction $\b{v}$ of a separating line of $K(N_q)$ and $B^{\delta}(N_q)$ such that $\b{v}\cdot \b{\hat{x}} \geq 0$ and $\b{v} \cdot \b{\hat{y}} \leq 0$. Then, we have (a) and (b) as in the preceding case, and 
\begin{enumerate}[(a)]
	\setcounter{enumi}{2}
	\item $S_{p}'$ is a line of positive slope below the curve $S_{p-1}$, and so the choice of $\b{v} \cdot \b{\hat{y}} \leq 0$ and $\b{v} \cdot \b{\hat{x}} \geq 0$ guarantees this intersection occurs.
	\item $S_{p}'$ has constant and strictly positive slope, so $\b{n}(S_{p}'(t))\cdot \b{v} >0$ for all $t>1$.
\end{enumerate}

If $S_- = S_p'$, we again have $B^{\delta}(N_{p}) = \bR^2_{>0}$ and so again can choose $\b{v}$ such that $\b{v}\cdot \b{\hat{x}} \geq 0$ and $\b{v} \cdot \b{\hat{y}} \leq 0$. If $S_{p}'$ and $S_{p}''$ are lines (note that they are either both lines or neither are lines) then, (a)-(d) are satisfied because $S_{p}'$ and $S_{p}''$ have positive slope while $L_t = S_p'(t) + \b{s}v$ has negative slope. Otherwise, (a) and (b) are the same as the preceding argument, while (d) is the same reflected across the line $y=x$ (and taking $-\b{v}$). Finally,
\begin{enumerate}[(a)]
	\setcounter{enumi}{2}
	\item If $\b{v} \cdot \b{\hat{x}} \geq 0$ and $\b{v} \cdot \b{\hat{y}} < 0$, then the line $L_t = S_{p}'(t) + s\b{v}$ intersects the (affine) half line $\{x>1, y=1\}$, which is below $S_p''$, making it clear that we have the desired intersection. If $\b{v} = \hat{\b{x}}$, we notice that for any point on $S_{p}'$, there is some point on $S_{p}''$ directly to the right, because $S_{p}''$ is of the form $y = \alpha x^m$ for $m>0$. 
\end{enumerate}

For $S_- \in \{S_{p}'',...,S_{m-1}\}$, we can reflect the previous arguments across the line $y=x$. Reflection gives a counterclockwise construction of each $L_t$ with directions that satisfy $\b{v}_i\cdot \b{\hat{y}}> 0$ or $\b{v} = \b{\hat{x}}$. A continuing clockwise construction would instead choose $-\b{v}$. This concludes the proof of \cref{cl2}.

\bigskip

\Cref{cl2} gives us a way to construct line segments that serve as a boundary of an invariant region for $\cK(\cN)$ within the individual regions between curves $S_{-}(t)$ and $S_{+}(t)$. Note that two such segments in adjacent regions do not necessarily have a common point. On the other hand, results (b),(c), and (d) of the claim are three of the four conditions of \cref{toftau}. Furthermore, to construct $H_2$ we only considered curves whose logarithmic image is not contained in the third quadrant, and so the last condition of \cref{toftau} is also satisfied. Therefore, by applying \cref{toftau} we can ``slide" each segment so that segments in adjacent regions meet at a point. Therefore, a solution of the differential inclusion $\cK(\cN)$ cannot cross $H_2$ in the outward direction.


We finally use \cref{toftau} to connect $H_1$, $H_2$, and the lines $\{x=x^*\}$, $\{y=y^*\}$. Together, these curves form the boundary $\partial \cR$ of a forward invariant region.

\bigskip

{\bf A nested, continuous family of regions.}

We have shown that we can build a forward invariant region $\cR$. We will now show that there exists a continuous nested family of such regions which cover $\bR^2_{>0}$.

Recall that the corner points used to construct $\cR$ lie on the curves $S_i(t) = (t^{m_1},\alpha t^{m_2})$, where $m_1,m_2$ and $\alpha$ depend on $i$. We have seen that there is some $t_0$ such that for $t>t_0$, the curve $S_i(t)$ is contained in $\mathit{fat}_{\fp}(N_{i})$.

To create a nested family of regions $\cR_{t}$, we first choose the corner $\b{x}_1$ of $\cR$ lying on $S_1$ and let $\hat{t}$ be such $S_1(\hat{t}) = \b{x}_1$. Likewise, label the corners of $\cR$ as $\b{x}_i$ and let $\b{v}_i = \b{x}_{i+1}-\b{x}_{i}$ as before. Next, we let $\cR=\cR_{\hat{t}}$ and for $t>\hat{t}$, $\cR_{t}$ is the region with $S_1(t)\in \partial \cR_t$ with sides which are segments 
\[
L_{t}^i = \{S_i(\tau_i^t) + s\b{v}_i| s \in [0,s_i^t] \}
\]
and corners lying on each $S_i(t)$. In addition, as we take $t\rightarrow \infty$, we take $x^* \rightarrow 0$ and $y^* \to 0$. \Cref{toftau} implies that all the corners of $\partial \cR_t$ moves outwards along the curves $S_i(t)$ as $t$ increases. Therefore, we obtain a nested family of regions (with disjoint boundaries) that covers $\bR^2_{>0}$. All of these regions satisfy the condition that $\b{w}\cdot \b{n}_i(\b{x}) \geq 0$ for $\b{w} \in K(N_i)$, where $\b{n}_i(\b{x})$ is the clockwise normal to the curve at $\b{x} \in \partial \cR_{t}$. Therefore, they are all forward invariant under the $\cN$-cone differential inclusion $\cK(\cN)$. They are then also forward invariant under any differential inclusion which is embedded in $\cK(\cN)$.
\end{proof}

%% file: defs_vk_rev.tex
\subsection{Definitions}
A \emph{variable $\ka$ polynomial dynamical system} (\emph{v$\ka$-polynomial dynamical system}) is a dynamical system on $\bR^d_{>0}$ which can be written
\begin{equation}\label{vkpls}
\dot{\b{x}} = \b{f}(\b{x},t) = \sum_{i=1}^{n} \ka_{i}(t)\b{x}^{\b{s}_i}\b{v}_{i}
\end{equation}
with $\b{x} \in \bR^d_{>0}$, $\b{s}_i,\b{v}_{i}\in \bR^d$ and $\b{x}^{\b{s}} = \prod_{l=1}^d x_l^{s_l}$ such that there is some $\varepsilon>0$ such that $\varepsilon\leq \ka_{i}(t)\leq \frac{1}{\varepsilon}$ for all $i$ and all $t>0$. In particular, polynomial dynamical systems with constant coefficients are a special case of v$\ka$-polynomial dynamical systems
.~\footnote{
Note that any polynomial dynamical system with constant coefficients can be written $\b{\dot{x}}~=~\sum_{i=1}^n\b{x}^{\b{s}_i}\b{v}_i$.
}

In order to investigate the geometric properties of a v$\ka$-polynomial dynamical system, we can use a \emph{Euclidean embedded graph}, as introduced in \cite{gheorgheGAC}.
\begin{definition}\label{GEGnotat}
A \emph{Euclidean embedded graph} $G = (\cV,\cE)$ is a finite directed graph whose nodes $\cV$ are labeled with distinct elements of a finite set $Y \subset \bR^d$.
\end{definition}
We define, for each edge $e \in \cE$, the \emph{source vector} $\b{s}(e)\in Y$ to be the label of the source node of $e$, the \emph{target vector} $\b{t}(e) \in Y$ to be the label of the target node of $e$, and the \emph{reaction vector} $\b{v}(e) = \b{t}(e) - \b{s}(e)$, to be the vector in $\bR^d$ from the label of the source to the label of the target. These definitions are inspired by the language of reaction network theory, with the source vector corresponding to the ``source complex" and the target vector corresponding the ``product complex" \cite{mfeinlec}\cite{cranaz}\cite{guna}\cite{hornandjackson}. Given a Euclidean embedded graph with edge set $\cE$, we can \emph{generate} the v$\ka$-polynomial dynamical system
\[
\dot{\b{x}} = \b{f}(\b{x},t) = \sum_{e\in \cE} \ka_e(t) \b{x}^{\b{s}(e)}\b{v}(e)
\]
by making an arbitrary choice of $\{\ka_{e}(t)\}$ that satisfies $\varepsilon\leq \ka_{e}(t)\leq \frac{1}{\varepsilon}$ for all $t$. If a v$\ka$-polynomial dynamical system can be constructed in such a way for some Euclidean embedded graph $G$, we then say that $\dot{\b{x}} = \b{f}(\b{x})$ is \emph{generated by $G$}. Notice that, depending on choice of $\ka_{e}(t)$, two different Euclidean embedded graphs $G$ and $G'$ may generate the same v$\ka$-polynomial dynamical system $\dot{\b{x}} = \b{f}(\b{x})$.

\begin{rmk} 
	While our analysis depends on node labels in the Euclidean embedded graph, we can sometimes obtain conclusions about the dynamics of a generated system using only information about the unlabeled graph, such as reversibility and weak reversibility \cite{hornandjackson}\cite{dave1}.
\end{rmk}

It will be useful to group terms in the sum \cref{vkpls} with the same exponent vectors $\b{s}(e) = \b{s}_i$. We can rewrite \cref{vkpls} as
\begin{equation}\label{vkpls2}
\dot{\b{x}} = \b{f}(\b{x},t) = \sum_{i=1}^n \b{x}^{\b{s}_i}\left(\sum_{j=1}^{m_i}\ka_{ij}(t)\b{v}_{ij}\right)
\end{equation}
by renumbering source and reaction vectors. Note that the value of $n$ in \cref{vkpls2} may be smaller than the value of $n$ in \cref{vkpls}.

Often of great interest in applications to biological and chemical modeling is whether a polynomial dynamical system satisfies a condition called \emph{permanence} or the weaker condition called \emph{persistence}. This paper is mainly concerned with permanence on $\bR^d_{> 0}$, the positive orthant. 

\begin{definition}\label{defperm2}
A d-dimensional dynamical system on $\bR^d_{> 0}$ is called \emph{permanent} if $\bR^d_{> 0}$ is forward invariant and there exists $\delta >0$ such that for any solution $\b{x}(t)$ with positive initial condition $\b{x}_0 \in \bR^d_{>0}$ we have
\[
\delta < \liminf_{t \rightarrow \infty}x_i(t) \quad \& \quad \limsup_{t \rightarrow \infty}x_i(t) < \frac{1}{\delta}
\]
for all $i \in \{1,...,d\}$.\cite{cranaz}
\end{definition}

Clearly, this condition implies persistence as introduced in \cref{defper}.

There is of course physical relevance to these conditions in chemical network models, and more generally in models of population dynamics. The question of permanence or persistence in a chemical setting is informally the question of whether or not some species in the network can be depleted.


We can also define a condition on v$\ka$-polynomial dynamical systems and euclidean embedded graphs similar to the tropically endotactic condition on differential inclusions.

\begin{definition}\label{tropendosys}
	Let $\b{\dot{x}} = \b{f}(\b{x},t)$ be a two dimensional v$\ka$-polynomial dynamical system. We say that $\b{\dot{x}} = \b{f}(\b{x},t)$ is \emph{tropically endotactic} if it is strictly embedded into a tropically endotactic differential inclusion on $\bR^2_{> 0}$. 
\end{definition}

\begin{definition}\label{tropendograph}
Let $G$ be a Euclidean embedded graph in $\bR^2$. We say that $G$ is \emph{tropically endotactic} if any v$\ka$-polynomial dynamical system generated by $G$ is tropically endotactic.
\end{definition}

\subsection{Permanence of tropically endotactic systems}
It is clear from \cref{thethm} that a two dimensional tropically endotactic system is persistent. We will prove also the stronger result that such systems are permanent. We show this by constructing a Lyapunov function outside of a compact attracting set, using the borders of the regions constructed in the proof of \cref{thethm} as level sets.

\begin{theorem}\label{realthm}
	Any two dimensional tropically endotactic v$\ka$-polynomial dynamical system is permanent.
\end{theorem}
\begin{proof}
	If $\b{\dot{x}}= \b{f}(\b{x},s)$ is tropically endotactic, then there exists a complete fan $\cN$ and tropically endotactic $\cN$-cone differential inclusion  $\cK(\cN)$ into which $\b{\dot{x}} = \b{f}(\b{x},s)$ is strictly embedded. For this differential inclusion, we construct a family of nested, forward invariant regions $\cR_t$ with disjoint boundaries such that they cover $\bR^2_{>0}$, as in the last step of the proof of \cref{thethm}. We do so by constructing one such region $\cR_{\hat{t}}$ and showing that there exists a nested continuous family of regions which contain $\cR_{\hat{t}}$. We may also assume that $\mathit{fat}_{\fp}(\b{1})\subset \cR_{\hat{t}}^{\circ}$. We will show that this family of invariant regions can be used to define a Lyapunov function for the system $\b{\dot{x}} = \b{f}(\b{x},s)$. Define   
	\begin{equation}
	\Lambda(\b{x})= \inf\{t \geq \hat{t}|\b{x} \in \cR_t\}
	\end{equation}
	Note that each region $\cR_t$ is polygonal, and $\Lambda$ is smooth everywhere except at the corner points of these polygons. Recall from the proof of \cref{thethm} that these corner points lie on curves denoted $S_i(t)$. Also, if $\Lambda(\b{x}(0)) = t^*$ and $\b{x}(s) \in \cR_{t^*}\setminus \cR_{\hat{t}}$, then $\|\dot{\b{x}}(s)\|>\zeta$ for some $\zeta>0$. This is because $\mathit{fat}_{\fp}(\b{1})\cap (\cR_{t^*}\setminus \cR_{\hat{t}})= \emptyset$, which implies that $\b{\dot{x}}$ belongs to the interior of a cone which is not the whole of $\bR^2$, so $\b{\dot{x}} \neq 0$ on the (compact) closure of $\cR_{t^*}\setminus \cR_{\hat{t}}$. 
	
	Choose some $\tilde{t} \in (\hat{t},t^*)$ and $\epsilon\in (0,\tilde{t}-\hat{t})$. We will show that if $\b{x}(s)$ is a solution of $\b{\dot{x}} = \b{f}(\b{x},s)$ with $\Lambda(\b{x}(0)) = t^*$, then $\b{x}(s)$ enters the forward invariant region $\cR_{\tilde{t}}$. We do this by showing that $\Lambda$ is a strict Lyapunov function when restricted to $\bR^2_{>0} \setminus \cR_{\hat{t}+\epsilon} =: (\cR_{\hat{t}+\epsilon})^c$. For this, we prove that if $\b{x}(s) \in \ol{\cR_{t^*}\setminus \cR_{\hat{t}+\epsilon}}$, then
	\begin{equation}\label{bound}
	\limsup_{s\rightarrow s_0}\frac{\Lambda(\b{x}(s))-\Lambda(\b{x}(s_0))}{s-s_0} < -\eta
	\end{equation}
	for some $\eta >0$. We do so by showing that $\nabla \Lambda \cdot \dot{\b{x}} < -\eta$ when $\b{x}(s_0)$ is a smooth point of $\Lambda$ and that
	\begin{equation}\label{neg}
	\limsup_{s\rightarrow s_0}\frac{\Lambda(\b{x}(s))-\Lambda(\b{x}(s_0))}{s-s_0} < -\eta
	\end{equation}
	when $\b{x}(s_0)$ in a neighborhood of some $S_i(t)$. Together, these imply that \cref{bound} holds.
	
	First consider any compact region contained in $\ol{\cR_{t^*}\setminus \cR_{\hat{t}+ \epsilon}}$ which does not intersect any $S_i(t)$. Note, $\Lambda$ is smooth in such regions and $\nabla \Lambda$ is precisely the outward normal to $\partial \cR_t$. The choice of edges of $\partial \cR_t$ as supporting lines of $K(N_i)$, along with the strict embedding of $\b{\dot{x}} = \b{f}(\b{x},s)$ into $\cK(\cN)$, guarantees that $\nabla \Lambda \cdot \dot{\b{x}} < 0$. Compactness of the region then ensures that $\nabla \Lambda \cdot \dot{\b{x}} < -\eta$ in this region for some $\eta >0$.
	
	To deal with the curves $\{S_i(t)|1<t<\infty\}$ along which $\Lambda$ is not differentiable, we must draw upon techniques of convex analysis, as detailed in \cite{varan}.
	
	For each curve $S_i$, let $\Lambda_{i1}$ and $\Lambda_{i2}$ be the smooth functions with constant gradient direction (and therefore with straight-line level sets) such that $\Lambda = \Lambda_{i1}$ on one side of $\{S_i(t)|1<t<\infty\}$ and $\Lambda = \Lambda_{i2}$ on the other and these functions extend onto a neighborhood of $S_i$. 
	
	In order to use some convex analysis results about lower $C^1$ functions, we will separate the proof into two cases. First, if the interior angle of each $\cR_t$ along the curve $S_i$ is less than or equal to $\pi$, then $\Lambda(\b{x}) = \max\{\Lambda_{i1}(\b{x}),\Lambda_{i2}(\b{x})\}$ in some neighborhood of the curve $S_i$, and so $\Lambda$ is lower $C^1$ \cite{varan}. Second, if the interior angle of each $\cR_t$ along the curve $S_i$ is greater than $\pi$, then $-\Lambda(\b{x}) = \max\{-\Lambda_{i1}(\b{x}),-\Lambda_{i2}(\b{x})\}$ and so $-\Lambda(\b{x})$ is lower $C^1$.
	
	 In the first case, consider a compact neighborhood $\cS_i$ of $\{S_i(t)|1\leq t \leq \infty\} \cap (\cR_{t^*}\setminus \cR_{\hat{t}})$ contained in some $\mathit{fat}_{\fp}(N_j)$ in which each $\cR_{t}$ is convex, and so $\Lambda$ is lower $C^1$. Recall that, for $t$ large enough, $S_i(t) \subset \mathit{fat}_{\fp}(N_j)$ for some $j$. Then the \emph{(general) subgradient} \cite{varan} of $\Lambda$ at $\b{x}$ along the curve $S_i(t)$ is the set
	\begin{equation}\label{subgrad}
	\partial \Lambda (\b{x}) = \{a\nabla \Lambda_{i1}(\b{x}) + (1-a)\nabla \Lambda_{i2}(\b{x})|a \in [0,1]\}
	\end{equation}
	 The function $\Lambda$ is strictly continuous in $\bR^2_{>0}$, so we can apply the chain rule for subgradients (Theorem 10.6 in \cite{varan}) to obtain that
	\begin{equation}
	\partial (\Lambda\circ \b{x})(s) \subseteq \{\b{w}\cdot \dot{\b{x}}(s)|\b{w}\in\partial \Lambda(\b{x}(s))\}
	\end{equation}
	From the construction of the regions $\cR_t$ and the fact that $\b{\dot{x}}$ must be contained in the strict interior of the cone $K(N_j)$ of the differential inclusion, we have $\b{\dot{x}} \cdot \nabla \Lambda_{i1} <0$ and $\b{\dot{x}} \cdot \nabla \Lambda_{i2} <0$ on $\cS_i$. This is because the edges of $\cR_t$ and so the level sets of $\Lambda_{i1}$, $\Lambda_{i2}$ were chosen to be supporting lines to $K(N_j)$. Since $\cS_i$ is compact, it follows that there is some $\eta>0$ such that $\b{\dot{x}} \cdot \nabla \Lambda_{i1} <-\eta$ and $\b{\dot{x}} \cdot \nabla \Lambda_{i2} <-\eta$ in $\cS_i$. Therefore, according to \cref{subgrad} there is a $\eta>0$ such that $\b{w}\cdot \dot{\b{x}}(s)<-\eta<0$ for all $\b{w}\in \partial \Lambda(\b{x}(s))$ in $\cS_i$.
	
	Because $\Lambda$ is lower $C^1$, from a generalized mean value theorem (Theorem 10.48 in \cite{varan}) applied to the function $(\Lambda \circ \b{x})(s)$ it follows that for all $s$ in some neighborhood of $s_0$ there is some $\tau_s\in [s_0,s]$ such that
	\begin{equation}
	\Lambda(\b{x}(s))- \Lambda(\b{x}(s_0)) = \sigma_s (s-s_0)\text{ for some scalar }\sigma_s \in \partial(\Lambda\circ \b{x})(\tau_s)
	\end{equation}
	But, we have seen that $\sigma_s < -\eta$. Therefore, in $\cS_i$, we have that 
	\[
	\limsup_{s\rightarrow s_0}\frac{\Lambda(\b{x}(s))-\Lambda(\b{x}(s_0))}{s-s_0} < -\eta
	\]
	
	In the second case, we have that the interior angle of each $\cR_t$ along the curve $S_i$ is greater than $\pi$. Again, we take this neighborhood to be contained in $\mathit{fat}_{\fp}(N_j)$. Then, completely analogously with the first case, we can consider the lower $C^1$ function $-\Lambda(\b{x})$ and we obtain that $-\Lambda$ is strictly increasing along trajectories $\b{x}(s)$. This implies that $\Lambda$ is strictly decreasing along trajectories $\b{x}(s)$, and moreover we have
	\[
	\limsup_{s\rightarrow s_0}\frac{\Lambda(\b{x}(s))-\Lambda(\b{x}(s_0))}{s-s_0} < -\eta
	\]	
	for some $\eta >0$.

	Finally, note that the three types of regions we have considered (in which $\Lambda$ is smooth, $\Lambda$ is lower $C^1$, or $-\Lambda$ is lower $C^1$) cover the entirety of $\ol{\cR_{t^*}\setminus \cR_{\hat{t}+\epsilon}}$. Therefore, we obtain that \cref{bound} holds on $\ol{\cR_{t^*}\setminus \cR_{\hat{t}+\epsilon}}$, and so $\Lambda$ decreases along trajectories $\b{x}(s)$ at a rate that is bounded away from $0$. Therefore, solutions to $\b{\dot{x}} = \b{f}(\b{x},s)$ must enter the forward invariant region $\cR_{\tilde{t}}$. 
\end{proof}

%% file: examples_rev.tex
\subsection{A modified Lotka-Volterra system}
It has been shown that two dimensional weakly reversible, or even \emph{endotactic} systems are permanent \cite{cranaz}.  \Cref{realthm} can be used to conclude permanence for systems which are neither weakly reversible nor endotactic. 

Consider the following modified version of the classical Lotka-Volterra predator-prey model,
\begin{equation}\label{mLV}
\b{\dot{x}} = \ka_1(t) x \begin{pmatrix}
1\\ \epsilon_1
\end{pmatrix} + \ka_2(t) xy \begin{pmatrix}
-1 \\ 1
\end{pmatrix} + \ka_3(t) y \begin{pmatrix}
\epsilon_2 \\ -1
\end{pmatrix}
\end{equation}
for $\b{x} = (x,y)^T \in \bR^2_{>0}$, such that $\epsilon_1,\epsilon_2 \in (0,1)$ and there exists $\varepsilon > 0$ with $\varepsilon < \ka_i(t) < \frac{1}{\varepsilon}$. Note that if $\epsilon_1 = \epsilon_2 =0$, the system \cref{mLV} becomes a variable $\ka$ version of the classical Lotka-Volterra model, and is \emph{not} persistent, and therefore not permanent \cite{cranaz}.

We will show that the system \cref{mLV} is permanent by embedding it into a tropically endotactic differential inclusion. This means we must find an exponential fan $\cD = \{D_i\}$ and cones $K(D_i)$ (and parameter $\fp$) such that $\b{x} \in \mathit{fat}_{\fp}(D_i) \Rightarrow \b{\dot{x}} \in K(D_i)^{\circ}$, and so that $K(D_i)$ does not intersect the $\delta$-escape directions $B^{\delta}(D_i)$ for some $\delta$ (see \cref{tropendo}). We will construct $\cD$ and the cones $K(D_i)$ by considering the relative magnitude of the three monomials $x$, $xy$, and $y$.
 
We construct a complete fan $\widehat{\cD}$ such that the ordering of these monomials is constant on the (relative) interiors of the regions $D_i$ of the exponential fan $\cD = \exp(\widehat{\cD})$. To do this, we use the three curves $x = xy$, $xy = y$, and $x = y$, which give rise to the one dimensional members of $\cD$. The regions $\mathit{fat}_{\fp}(D_i)$ are shown in \cref{egraphmLV}(b) and (c) in blue and white. We can calculate the cones $B^{\delta}(D_i)$ by writing $\fp$-escape curves as 
\begin{equation}\label{curve_calc1}
C(t) = \exp\left(\b{r}t + g(t) \b{p}\right) = \exp(r_1 t + g(t) p_1)\b{\hat{x}} + \exp(r_2 t + g(t) p_2)\b{\hat{y}}
\end{equation}
where $g(t) = 1- \alpha e^{-\beta t}$, $\beta \geq 0$, and $\alpha \in \bR$.

For example, consider the three cones such that $(1,1) \in \widehat{D}_i$ \footnote{$\mathit{fat}_{\fp}(D_0)$ shown in blue, $\mathit{fat}_{\fp}(D_{-1})$ and $\mathit{fat}_{\fp}(D_{1})$ shown in white in the upper right of \cref{egraphmLV} (c)}. Take $\b{r} = (1,1)$ and $\b{p} = (-p,p)$, so a $\fp$-escape curve can be written
\begin{equation}\label{curve_calc2}
C(t) = e^{t + g(t)(-p)}\b{\hat{x}} + e^{t + g(t)p}\b{\hat{y}}
\end{equation}
and so
\begin{equation}\label{curve_calc3}
\left[C'(t)\right]_n = \left[(1 - g'(t)p)e^{t - g(t)p}\b{\hat{x}} + (1 + g'(t)p)e^{t + g(t)p}\b{\hat{y}}\right]_n
\end{equation}
We then multiply by the scalar $e^{-t}$ to see that
\begin{equation}\label{curve_calc4}
\left[C'(t)\right]_n = \left[(1 - g'(t)p)e^{-g(t)p}\b{\hat{x}} + (1 + g'(t)p)e^{g(t)p}\b{\hat{y}}\right]_n
\end{equation}
and so
\begin{equation}\label{curve_calc5}
\lim_{t\rightarrow \infty}\left[C'(t)\right]_n = \left[e^{-p}\b{\hat{x}} + e^{p}\b{\hat{y}}\right]_n
\end{equation}
In the region $\mathit{fat}_{\fp}(D_0)$ in \cref{egraphmLV} (c) (shown in blue), it is true that for any choice of $p$, there is small enough $\fp$ so that $C(t)$ as above is a $\fp$-escape curve (see \cref{logC}), and so $B^{\delta}(D_0)$ approaches $\mathit{Cone}(\b{\hat{x}},\b{\hat{y}})$. Considering $D_{-1}$, as $\fp$ approaches $0$, we must have $p$ approaching $+\infty$ (and so $[C'(t)]_n$ approaching $\b{\hat{y}}$) in order to ensure that $C(t) \subset \mathit{fat}_{\fp}(D_{-1})$ for large $t$. Checking other possibilities of $\b{r} \in \widehat{D}_{-1}$ reveals that $B^{\delta}(D_{-1})$ is as shown in \cref{egraphmLV} (c).
\begin{figure}
	\begin{center}
	\begin{tikzpicture}[scale = 1.2]
		\draw[->] (0,0)--(3,0);
		\draw[->] (0,0)--(0,3);
		\draw (0,0) -- (3,3) node[pos = 1,right] {$\widehat{D}_0$};
		\draw[-latex, thick] (0,0) -- (1,1) node[pos = 0.8,below right] {$\b{r}$};
		\draw[dashed] (0,0.9)--(2.1,2.9);
		\draw[-latex, thick] (1,1)--(0.6,1.4) node[pos = 0.8,below left] {$\b{p}$};
		\fill[color = blue, opacity = 0.3] (0.5,1.5)--(2,3) to [out = -20, in = 160] (2.8,2.8) to [out = -20, in = 160] (3,2)--(1.5,0.5) --(0.5,0.5);
		\draw[thick] (1,1.2) to [out = 70, in = 225] (2.1,2.9) node[left] {$\log(C(t))$};
	\end{tikzpicture}
	\end{center}
\caption{The cone $\widehat{D}_0$, the region $\mathit{fat}_{\fp}(\widehat{D}_0)$ (blue) and $\log(C(t))$ for a $\fp$-escape curve $C$. Notice that $\{\b{r}t + \b{p}|t>t_0\}\subset \mathit{fat}_{\fp}(\widehat{D}_0)$ if $\|\b{p}\|< |\log(\fp)|$, while $\{\b{r}t + \b{p}|t>t_0\}\subset \mathit{fat}_{\fp}(\widehat{D}_{-1})$ if $\|\b{p}\|> |\log(\fp)|$.}\label{logC}
\end{figure}

In choosing the cones $K(D_i)$, the key observation is that for small $\fp$, the largest monomial is \emph{much} larger than the others when $(x,y)^T \in \mathit{fat}_{\fp}(D_i)$, and so this term ``dominates" the sum \cref{mLV}. For example, when $x \gg xy \gg y$, then 
\[
[\b{\dot{x}}]_n \approx \left[\begin{pmatrix}
1\\ \epsilon_1
\end{pmatrix}\right]_n
\]
and so in this region\footnote{The white region on the bottom right side of \cref{egraphmLV} (b).} we take $K(D_i)$ to be a cone of directions close to the direction $(1,\epsilon_1)^T$. Furthermore, because $xy$ is the second largest monomial, we can take $K(D_i)$ to be only vectors to the counterclockwise side of $(1,\epsilon_1)^T$, the same side as $(-1,1)$. When two or more dominant monomials are of the same order of magnitude, we define $K(D_i)$ to be the cone generated by their associated reaction vectors\footnote{As in the blue region in the right side of \cref{egraphmLV} (b), where $x$, $xy$ are the dominant monomials, and their orders of magnitude are the same.}. We can conclude that \cref{mLV} is permanent for any allowable choice of $\ka_i(t)$. This system is not endotactic, and in particular not weakly reversible.

Note that the differential inclusion constructed in this way is \emph{not} tropically endotactic for $\epsilon_2 \geq 1$, and indeed the system \cref{mLV} is not permanent in that case. The differential inclusion remains tropically endotactic for $\epsilon_1 \geq 1$, and so the system \cref{mLV} is permanent if $\epsilon_1\geq 1$.

\begin{figure}
	\begin{subfigure}[t]{0.32\textwidth}
		\begin{center}
			\begin{tikzpicture}[scale = 1.6]
			\draw[->] (-0.1,0)--(2,0);
			\draw[->] (0,-0.1)--(0,2);
			\shade[shading=ball, ball color=blue] (1,0) circle(.05);
			\draw[-{Latex[length=2mm,width=2mm]}] (1,0)--(2,0.2);
			\shade[shading=ball, ball color=blue] (2,0.2) circle(.05);
			\shade[shading=ball, ball color=blue] (1,1) circle(.05);
			\draw[-{Latex[length=2mm,width=2mm]}] (1,1)--(0,2);
			\shade[shading=ball, ball color=blue] (0,2) circle(.05);
			\shade[shading=ball, ball color=blue] (0,1) circle(.05);
			\draw[-{Latex[length=2mm,width=2mm]}] (0,1)--(0.2,0);
			\shade[shading=ball, ball color=blue] (0.2,0) circle(.05);
			\end{tikzpicture}
		\end{center}
	\captionsetup{width=0.9\textwidth}
	\caption{Euclidean embedded graph which generates the system \cref{mLV}}
	\end{subfigure}
	\begin{subfigure}[t]{0.32\textwidth}
	\begin{center}
		\begin{tikzpicture}[scale = 1.6]
			\draw[->] (-0.1,0)--(2,0);
			\draw[->] (0,-0.1)--(0,2);
			\fill[color = blue, opacity = 0.5] (0,0)--(2,1.9)--(2,2)--(1.9,2);
			\fill[color = blue, opacity = 0.5] (0.9,0)--(0.9,2)--(1.1,2)--(1.1,0);
			\fill[color = blue, opacity = 0.5] (0,0.9)--(2,0.9)--(2,1.1)--(0,1.1);
			\draw[-latex, color = orange] (1.3,0.15)--(1.8,0.25);
			\fill[opacity = 0.5, color = orange] (1.3,0.15)--(1.8,0.25)--(1.8,0.3);
			\draw[-latex, color = orange] (0.3,0.15)--(0.8,0.25);
			\fill[opacity = 0.5, color = orange] (0.3,0.15)--(0.8,0.25)--(0.8,0.2);
			\draw[-latex, color = orange] (0.15,1.8)--(0.25,1.3);
			\fill[opacity = 0.5, color = orange] (0.15,1.8)--(0.25,1.3)--(0.2,1.3);
			\draw[-latex, color = orange] (0.15,0.8)--(0.25,0.3);
			\fill[opacity = 0.5, color = orange] (0.15,0.8)--(0.25,0.3)--(0.3,0.3);
			\draw[-latex, color = orange] (1.9,1.2)--(1.6,1.5);
			\fill[opacity = 0.5, color = orange] (1.9,1.2)--(1.6,1.5)--(1.63,1.53);
			\draw[-latex, color = orange] (1.5,1.6)--(1.2,1.9);
			\fill[opacity = 0.5, color = orange] (1.5,1.6)--(1.2,1.9)--(1.17,1.87);
			\draw[-latex, color = orange] (1.7,1.7)--(1.5,1.9);
			\fill[opacity = 0.5, color = orange] (1.47,1.87)--(1.7,1.7)--(1.53,1.93);
			\draw[-latex, color = orange] (1,0.4)--(1.25,0.45);
			\fill[opacity = 0.5, color = orange] (1.22,0.42)--(1,0.4)--(1.28,0.48);
			\draw[-latex, color = orange] (0.4,1)--(0.45,0.75);
			\fill[opacity = 0.5, color = orange] (0.42,0.72)--(0.4,1)--(0.48,0.78);
			\draw[latex-latex, color = orange](0.75,0.55)--(0.5,0.5)--(0.55,0.25);
			\fill[opacity = 0.5, color = orange] (0.75,0.55)--(0.5,0.5)--(0.55,0.25);
			\draw[latex-latex, color = orange](1.75,1.05)--(1.5,1)--(1.35,1.15);
			\fill[opacity = 0.5, color = orange] (1.75,1.05)--(1.5,1)--(1.35,1.15);
			\draw[latex-latex, color = orange](0.85,1.65)--(1,1.5)--(1.05,1.25);
			\fill[opacity = 0.5, color = orange] (0.85,1.65)--(1,1.5)--(1.05,1.25);
		\end{tikzpicture}
	\end{center}
	\captionsetup{width=0.9\textwidth}
	\caption{$\cD$-cone differential inclusion $\cK(\cD)$, shown in orange.}
	\end{subfigure}
	\begin{subfigure}[t]{0.32\textwidth}
	\begin{center}
		\begin{tikzpicture}[scale = 1.6]
		\draw[->] (-0.1,0)--(2,0);
		\draw[->] (0,-0.1)--(0,2);
		\fill[color = blue, opacity = 0.5] (0,0)--(2,1.9)--(2,2)--(1.9,2);
		\fill[color = blue, opacity = 0.5] (0.9,0)--(0.9,2)--(1.1,2)--(1.1,0);
		\fill[color = blue, opacity = 0.5] (0,0.9)--(2,0.9)--(2,1.1)--(0,1.1);
		\draw[color = red, latex-latex] (1.6,2)--(1.6,1.6)--(2,1.6);
		\fill[color = red, opacity = 0.5]  (1.6,2)--(1.6,1.6)--(2,1.6);
		\draw[color = red, -latex] (1.3,1.6)--(1.3,2);
		\fill[color = red, opacity = 0.5] (1.3,1.6)--(1.3,2)--(1.4,2);
		\draw[color = red, -latex] (1,1.5)--(1,1.9);
		\fill[color = red, opacity = 0.5] (0.95,1.9)--(1,1.5)--(1.05,1.9);
		\draw[color = red, latex-latex] (0.5,1.6)--(0.5,2);
		\fill[color = red, opacity = 0.5] (0.5,1.6)--(0.5,2)--(0.4,2)--(0.4,1.6);
		\draw[color = red, latex-latex] (0.3,1.1)--(0.3,0.9);
		\fill[color = red, opacity = 0.5] (0.3,1.1)--(0.3,0.9)--(0.1,0.9)--(0.1,1.1);
		\draw[color = red, latex-latex] (0.2,0.8)--(0.2,0.4);
		\fill[color = red, opacity = 0.5] (0.2,0.8)--(0.2,0.4)--(0.1,0.4)--(0.1,0.8);
		\draw[color = red, latex-latex] (0.1,0.35)--(0.35,0.35)--(0.35,0.1);
		\fill[color = red, opacity = 0.5] (0.1,0.35)--(0.35,0.35)--(0.35,0.1);
		\draw[color = red, latex-latex] (0.8,0.2)--(0.4,0.2);
		\fill[color = red, opacity = 0.5] (0.8,0.2)--(0.4,0.2)--(0.4,0.1)--(0.8,0.1);
		\draw[color = red, latex-latex] (0.9,0.3)--(1.1,0.3);
		\fill[color = red, opacity = 0.5] (0.9,0.3)--(1.1,0.3)--(1.1,0.1)--(0.9,0.1);
		\draw[color = red, -latex] (1.6,1.3)--(2,1.3);
		\fill[color = red, opacity = 0.5] (1.6,1.3)--(2,1.3)--(2,1.4);
		\draw[color = red, -latex] (1.5,1)--(1.9,1);
		\fill[color = red, opacity = 0.5] (1.9,0.95)--(1.5,1)--(1.9,1.05);
		\draw[color = red, latex-latex] (1.6,0.5)--(2,0.5);
		\fill[color = red, opacity = 0.5] (1.6,0.5)--(2,0.5)--(2,0.4)--(1.6,0.4);
		\node at (2.1,2.1) {$D_0$};
		\node at (1.4,2.1) {$D_{-1}$};
		\node at (2.2,1.4) {$D_1$};
		\end{tikzpicture}
	\end{center}
	\captionsetup{width=0.9\textwidth}
	\caption{Cones $B^{\delta}(D_i)$ for the fan $\cD$, shown in red.}
\end{subfigure}
\caption{Analysis of the system \cref{mLV}. Inspection of (b) and (c) shows that the $\cD$-cone differential inclusion, shown in (b), is tropically endotactic, because for small enough $\delta$, none of the cones $K(D_i)$ of the differential inclusion shown in orange in (b) intersect the interior of the corresponding cone $B^{\delta}(D_i)$ shown in red in (c).}\label{egraphmLV}
\end{figure}
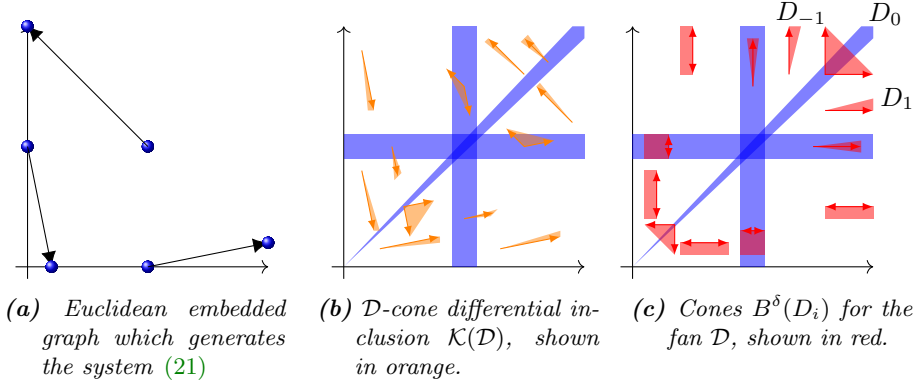

\subsection{An application to a chemical reaction system}
Consider the reaction network
\begin{equation}\label{5d}
\xymatrix @R = .5pc{
	Z+Y \ar[r]^{\ka_1} & 2X + 2Y\\
	W+X \ar[r]^{\ka_2} & X+Y \ar[r]^{\ka_3} & V
}
\end{equation}
If we assume that the rates of these reactions are given by a combination of mass action and Michaelis-Menten kinetics \cite{ekesh}, and in addition we know that there is some $\varepsilon>0$ such that $\varepsilon\leq W, Z \leq \frac{1}{\varepsilon}$, then we obtain the following dynamical system for the concentrations of $X$ and $Y$:

\begin{equation}\label{notpoly}
\frac{d}{dt}\begin{pmatrix}
x\\
y
\end{pmatrix} = \ka_1(t)y \begin{pmatrix}
2\\1
\end{pmatrix}
+ \frac{\ka_2(t)}{1+xy} x \begin{pmatrix}
0\\1
\end{pmatrix}
+ \ka_3(t) xy \begin{pmatrix}
-1\\ -1
\end{pmatrix}
\end{equation}

The system \cref{notpoly} is not a polynomial dynamical system due to the rational second term. However, multiplication by a positive scalar field does not change the permanence or persistence properties of a system. Therefore, we can replace \cref{notpoly} with the polynomial dynamical system \cref{rrobsys}, obtained by multiplying the right-hand side of \cref{notpoly} by the scalar field $(1+xy)$.

\begin{multline}\label{rrobsys}
\frac{d}{dt}\begin{pmatrix}
x\\
y
\end{pmatrix} = \ka_1(t)y \begin{pmatrix}
2\\1
\end{pmatrix}
+ \ka_1(t)xy^2 \begin{pmatrix}
2\\1
\end{pmatrix}
+ \ka_2(t) x \begin{pmatrix}
0\\1
\end{pmatrix}
\\+ \ka_3(t) xy \begin{pmatrix}
-1\\ -1
\end{pmatrix}
+ \ka_3(t) x^2y^2 \begin{pmatrix}
-1\\ -1
\end{pmatrix}
\end{multline}

It is not immediately apparent that this system should be permanent, or even that it should have bounded trajectories. Let $G$ be the Euclidean embedded graph shown in \cref{egraphex2} (a). Then $G$ generates the system \cref{rrobsys}. Again, we choose a fan based on a comparison of the relative magnitudes of the monomials. For this example, we only consider the which monomial is largest, rather than an ordering of the monomials as in the previous example. We construct the exponential fan $\cP$ by choosing regions on which the largest monomial does not change. Notice that $\widehat{\cP} = \log(\cP)$ is then the normal fan \cite{polytopes} for the convex hull of the source labels of $G$ ($\b{s}_i$ in figure \cref{egraphex2} (a)) making this fan simple to construct. We use cones $K(P_i)$ of directions which are close to the direction of the reaction vector associated with the largest monomial. The $\cP$-cone differential inclusion constructed in this way is tropically endotactic, and so we can conclude that the system \cref{rrobsys} is permanent, and so the system \cref{notpoly} is permanent as well.

Using a complete ordering of monomials, as we did for the system \cref{mLV} may allow us to choose smaller cones $K(D_i)$, and so has higher chances in general of resulting in a tropically endotactic differential inclusion. However, in this case it is sufficient to use the simpler exponential fan $\cP$.

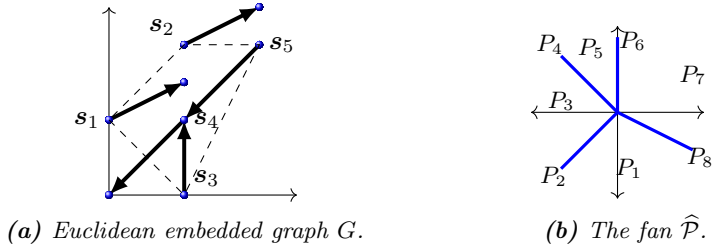
\begin{figure}
	\begin{center}
\begin{subfigure}[b]{0.45 \textwidth}
	\begin{center}
\begin{tikzpicture}[scale = 1]
	\draw[->] (-0.05,0) -- (2.5,0);
	\draw[->] (0,-0.05) -- (0,2.5);
	\draw[-latex,line width = 1.5pt] (0,1)--(1,1.5);
	\shade[shading=ball, ball color=blue] (0,1) circle(.05) node[left] {$\b{s}_1$};
	\shade[shading=ball, ball color=blue] (1,1.5) circle(.05);
	\draw[-latex,line width = 1.5pt] (1,2)--(2,2.5);
	\shade[shading=ball, ball color=blue] (1,2) circle(.05) node[above left] {$\b{s}_2$};
	\shade[shading=ball, ball color=blue] (2,2.5) circle(.05);
	\draw[-latex,line width = 1.5pt] (1,0)--(1,1);
	\shade[shading=ball, ball color=blue] (1,0) circle(.05) node[above right] {$\b{s}_3$};
	\draw[-latex,line width = 1.5pt] (2,2)--(1,1);
	\shade[shading=ball, ball color=blue] (2,2) circle(.05) node[right] {$\b{s}_5$};
	\draw[-latex,line width = 1.5pt] (1,1)--(0,0);
	\shade[shading=ball, ball color=blue] (1,1) circle(.05) node[right] {$\b{s}_4$};
	\shade[shading=ball, ball color=blue] (0,0) circle(.05);
	\draw[dashed] (0,1)--(1,0)--(2,2)--(1,2) -- cycle;
\end{tikzpicture}
\end{center}
\caption{Euclidean embedded graph $G$.}
\end{subfigure}
\begin{subfigure}[b]{0.45\textwidth}
	\begin{center}
		\begin{tikzpicture}[scale = 0.5]
		\draw[<->] (0,-2.3) -- (0,2.3);
		\draw[<->] (-2.3,0) -- (2.3,0);
		\draw[color = blue, very thick] (0,0)--(-1.5,-1.5);
		\draw[color = blue, very thick] (0,0)--(-1.5,1.5);
		\draw[color = blue, very thick] (0,0)--(0,2);
		\draw[color = blue, very thick] (0,0)--(2,-1);
		\node at(0.3,-1.5) {\small $P_1$};
		\node at(-1.7,-1.7) {\small$P_2$};
		\node at(-1.5,0.3) {\small$P_3$};
		\node at(-1.8,1.8) {\small$P_4$};
		\node at(-0.7,1.7) {\small$P_5$};
		\node at(0.4,1.9) {\small$P_{6}$};
		\node at(2,1) {\small$P_{7}$};
		\node at(2.2,-1.2) {\small$P_{8}$};
		\end{tikzpicture}
	\end{center}
	\caption{The fan $\widehat{\cP}$.}
\end{subfigure}
\end{center}
\caption{(a) Euclidean embedded graph that generates the polynomial dynamical system \cref{rrobsys}, with the convex hull of the source nodes shown dotted. (b) The normal fan $\widehat{\cP}$ for the dotted polygon in (a). The exponential fan $\cP = \exp(\widehat{\cP})$ gives regions in which the largest monomial does not change.}\label{egraphex2}
\end{figure}

To see that the polynomial dynamical system \cref{rrobsys} is tropically endotactic, we can check each cone $K(P_{i})$ and $B^{\delta}(P_{i})$ for some small $\delta$. A sample of the relevant analysis is demonstrated in \cref{permsyspartition}.
\begin{figure}[htp]
\begin{center}
\begin{subfigure}[b]{0.3\textwidth}
\begin{center}
\begin{tikzpicture}[scale = 1]
	\path[fill,color =red, opacity=0.4] (0,1)--(2,1)--(2,0)--(0,0)--(0,1);
	\draw[latex-latex, very thick, color = red] (0,1)--(2,1);
	\node at (1,0.5) {\small $B^{\delta}(P_1)$};
	\path[fill, opacity=0.4] (1,1)--(1.1,2)--(0.9,2)--(1,1);
	\draw[-latex, very thick] (1,1)--(1,2) node[pos = 1, above] {\small $K(P_1)$};
\end{tikzpicture}
\end{center}
\caption{}
\end{subfigure}\qquad
\begin{subfigure}[b]{0.3\textwidth}
\begin{center}
\begin{tikzpicture}[scale = 1]
	\path[fill,color =red, opacity=0.4] (0,1)--(1,1)--(1,0)--(0,0)--(0,1);
	\draw[latex-latex, color = red, very thick] (0,1)--(1,1)--(1,0);
	\node at (0.5,0.5) {\small $B^{\delta}(P_{2})$};
	\draw[latex-latex, very thick] (2,1.5)--(1,1) node[pos = 0.5,below right] {\small $K(P_{2})$}--(1,2) ;
	\path[fill,opacity = 0.4] (2,1.5)--(1,1)--(1,2)--(2,1.5);
\end{tikzpicture}
\end{center}
\caption{}
\end{subfigure}
\begin{subfigure}[b]{0.3\textwidth}
\begin{center}
\begin{tikzpicture}[scale = 1]
	\path[fill,color =red, opacity=0.4] (1,1)--(1,2)--(1.1,2)--(1,1);
	\draw[-latex, color = red, very thick] (1,1)--(1,2) node[pos = 1, left] {\textcolor{black}{\small $B^{\delta}(P_{7})$}};
	\path[fill, opacity=0.4] (1,1)--(0,0.1)--(0.1,0) -- (1,1);
	\draw[-latex,very thick] (1,1)--(0,0) node[pos = 1, above left] {\small $K(P_{7})$};
\end{tikzpicture}
\end{center}
\caption{}
\end{subfigure}
\end{center}
\caption{Permanence of \cref{rrobsys} can be concluded by determining the intersection of the cones $K(P_i)$ and $B^{\delta}(P_i)^{\circ}$. For every pair, this intersection must be empty. We show here a sample of the analysis. The remaining cones can be checked in the same way.}
\label{permsyspartition}
\end{figure}
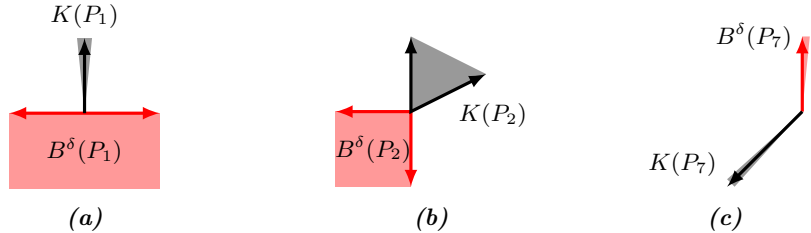

\subsection{Weakly reversible polynomial dynamical systems}

We can use \cref{realthm} to show that any weakly reversible system in two dimensions is permanent. A \emph{weakly reversible system} is a v$\ka$-polynomial dynamical system generated by a Euclidean embedded graph $G$ such that every edge of $G$ is contained in a (directed) cycle. We can show that a weakly reversible system is tropically endotactic using $\cT$, the fan of the toric differential inclusion used in \cite{gheorgheGAC}, Theorem 3.1. That theorem is
\begin{theorem}[Theorem 3.1 of \cite{gheorgheGAC} and Theorem 4.1 of \cite{gheorgheToricDI}]\label{toric}
Any weakly reversible v$\ka$-polynomial dynamical system can be embedded into a toric differential inclusion.
\end{theorem}
Note that in \cite{gheorgheGAC}, weakly reversible v$\ka$-polynomial dynamical systems are called ``k-variable toric dynamical systems" (see page 7 of \cite{gheorgheGAC}). A toric differential inclusion with exponential fan $\cT$ assigns to each region $\mathit{fat}_{\fp}(T_i)$ the negative dual, or \emph{polar cone}, $-\widehat{T}^*_i$ of the cone $\widehat{T}_i$ of the polyhedral fan $\widehat{\cT}$. Following the proof of \cref{toric} in \cite{gheorgheGAC}, it can be shown that when the v$\ka$-polynomial dynamical system is two dimensional and has no linear conserved quantities, the embedding implied above is strict. 

\begin{lemma}
Any toric differential inclusion in $\bR^2_{>0}$ is tropically endotactic.
\end{lemma}

\begin{proof}
Let $\cF$ be a toric differential inclusion. $\cF$ is a $\cT$-cone differential inclusion with the property that if $N\in \cT$ is a face of $M\in \cT$, then $K(M)\subseteq K(N)$. We need only show that for every $N\in \cT$ besides $\b{1}$, 
\[K(N) \cap B^{\delta}(N)^{\circ} = \emptyset\]
The construction of $B^{\delta}(N)$ implies that $B^{\delta}(N) \subseteq \widehat{N}^*$ or $\widehat{N}^* \subseteq B^{\delta}(N)$ (see \cref{escapedirsfig}), where $\widehat{N}^*$ is the dual cone, while $K(N) = -\widehat{N}^*$. If $B^{\delta}(N) \subseteq \widehat{N}^*$, then $B^{\delta}(N)^{\circ}$ clearly does not intersect $-\widehat{N}^*$. If $\widehat{N}^* \subseteq B^{\delta}(N) \subset H$ where $H$ is a half-plane, then the line $\partial H$ is a supporting line of both $-\widehat{N}^*$ (because it is a supporting line of $\widehat{N}^*$) and $B^{\delta}(N)$. The line $\partial H$ must also separate $B^{\delta}(N)$ and $-\widehat{N}^*$, because it does not separate $B^{\delta}(N)$ and $\widehat{N}^*$.
\end{proof}

We then obtain the following:
\begin{corollary}
Any weakly reversible v$\ka$-polynomial dynamical system in $\bR^2_{>0}$ with no linear conserved quantities is tropically endotactic.
\end{corollary}

This result can be used to prove the \emph{global attractor conjecture} in three dimensions, as in \cite{cranaz}.

%% file: furture_work_rev.tex
We will in upcoming work introduce an algorithmic construction of an $\cN$-cone differential inclusion, which we call the \emph{dominance differential inclusion}, into which a given polynomial dynamical system is embedded. In fact, given a Euclidean embedded graph $G$ and an exponential fan $\cN$, we can construct the dominance differential inclusion $\cD_G(\cN)$ such that for any polynomial dynamical system generated by $G$, there is some $\fp$ such that the system is strictly embedded in $\cD_G(\cN)$. The dominance differential inclusion was used to show permanence of examples \cref{mLV} and \cref{rrobsys}. We conjecture that this construction is minimal, in the sense that if a polynomial dynamical system can be strictly embedded into some tropically endotactic differential inclusion, then $\cD_{G}$ itself must be tropically endotactic.

We have shown in this paper that if a v$\ka$-polynomial dynamical system is tropically endotactic, it is permanent. On the other hand, we have found examples of v$\ka$-polynomial dynamical systems that are permanent and fail to be tropically endotactic, so the property of being tropically endotactic is not necessary and sufficient for permanence. However, in future work we will show that this property is closely related to a necessary condition for permanence.

The definition of tropically endotactic differential inclusions can be extended to higher dimensions, and in future work we will show that it gives rise to a necessary condition for permanence in any dimension.

%% file: acks_rev.tex
The authors have received partial support from NSF-DMS-1412643.